\documentclass[10pt,letterpaper]{article}

\author{Benjamin Antieau and Elden Elmanto}

\usepackage[letterpaper,left=1.5in,right=1.5in,top=1.5in,bottom=1.5in]{geometry}

\newcommand{\myauthor}{Benjamin Antieau and Elden Elmanto}
\newcommand{\mytitle}{Descent for semiorthogonal decompositions}
\newcommand{\pdftitle}{\mytitle}

\title{Descent for semiorthogonal decompositions}

\usepackage[pdfstartview=FitH,
            pdfauthor={\myauthor},
            pdftitle={\pdftitle},
            colorlinks,
            linkcolor=reference,
            citecolor=citation,
            urlcolor=e-mail,
            backref]{hyperref}

\usepackage{amsmath}
\usepackage{amscd}
\usepackage{amsbsy}
\usepackage{amssymb}
\usepackage{verbatim}
\usepackage{eufrak}
\usepackage{eucal}
\usepackage{microtype}
\usepackage{amsthm}
\usepackage{xspace}
\usepackage[all,cmtip]{xy}
\usepackage{tikz}
\usetikzlibrary{matrix,arrows}
\usepackage{dsfont}

\usepackage{tikz-cd}

\usepackage{fancyhdr}
\usepackage{makeidx}
\makeindex
\newcommand{\df}[1]{{\bf #1}\index{#1}\xspace}

\usepackage{color}
\definecolor{todo}{rgb}{1,0,0}
\definecolor{conditional}{rgb}{0,1,0}
\definecolor{e-mail}{rgb}{0,.40,.80}
\definecolor{reference}{rgb}{.20,.60,.22}
\definecolor{mrnumber}{rgb}{.80,.40,0}
\definecolor{citation}{rgb}{0,.40,.80}

\let\oldmarginpar\marginpar
\renewcommand\marginpar[1]{\-\oldmarginpar[\raggedleft\footnotesize #1]%
{\raggedright\footnotesize #1}}

\newcommand{\Ascr}{\mathcal{A}}

\newcommand{\Cscr}{\mathcal{C}}
\newcommand{\Dscr}{\mathcal{D}}
\newcommand{\Escr}{\mathcal{E}}
\newcommand{\Fscr}{\mathcal{F}}
\newcommand{\Gscr}{\mathcal{G}}

\newcommand{\Mscr}{\mathcal{M}}

\newcommand{\Oscr}{\mathcal{O}}
\newcommand{\Pscr}{\mathcal{P}}

\newcommand{\Sscr}{\mathcal{S}}

\newcommand{\Vscr}{\mathcal{V}}

\newcommand{\D}{\mathrm{D}}
\newcommand{\E}{\mathrm{E}}
\newcommand{\F}{\mathrm{F}}
\newcommand{\G}{\mathrm{G}}
\renewcommand{\H}{\mathrm{H}}
\newcommand{\K}{\mathrm{K}}
\renewcommand{\L}{\mathrm{L}}
\newcommand{\M}{\mathrm{M}}
\newcommand{\N}{\mathrm{N}}

\newcommand{\R}{\mathrm{R}}

\newcommand{\CC}{\mathds{C}}

\newcommand{\EE}{\mathds{E}}

\newcommand{\GG}{\mathds{G}}

\newcommand{\NN}{\mathds{N}}
\newcommand{\PP}{\mathds{P}}

\newcommand{\RR}{\mathds{R}}
\renewcommand{\SS}{\mathds{S}}

\newcommand{\ZZ}{\mathds{Z}}

\newcommand{\dg}{\mathrm{dg}}

\newcommand{\op}{\mathrm{op}}
\newcommand{\ex}{\mathrm{ex}}

\newcommand{\perf}{\mathrm{perf}}

\newcommand{\Sp}{\mathcal{S}\mathrm{p}}

\newcommand{\ExFilt}{\mathrm{ExFilt}}
\newcommand{\Filt}{\mathrm{Filt}}

\DeclareMathOperator{\id}{id}

\newcommand{\cofib}{\mathrm{cofib}}
\newcommand{\fib}{\mathrm{fib}}

\renewcommand{\geq}{\geqslant}
\renewcommand{\leq}{\leqslant}
\newcommand{\Ho}{\mathrm{Ho}}

\newcommand{\Cat}{\mathrm{Cat}}

\newcommand{\PrL}{\mathrm{Pr}^\L}

\newcommand{\HH}{\mathrm{HH}}

\DeclareMathOperator{\Pic}{Pic}

\DeclareMathOperator{\Fun}{Fun}

\DeclareMathOperator{\Hom}{Hom}
\newcommand{\Map}{\mathrm{Map}}
\newcommand{\MapSp}{\mathbf{Map}}

\newcommand{\ShAut}{\mathbf{Aut}}

\newcommand{\ShMap}{\mathbf{Map}}

\newcommand{\Ex}{\mathrm{Ex}}

\newcommand{\Mod}{\mathrm{Mod}}
\newcommand{\Modscr}{\Mscr\mathrm{od}}

\newcommand{\Perf}{\mathrm{Perf}}
\newcommand{\Perfscr}{\Pscr\mathrm{erf}}

\newcommand{\Ind}{\mathrm{Ind}}

\newcommand{\CAlg}{\mathrm{CAlg}}

\newcommand{\Aff}{\mathrm{Aff}}
\newcommand{\Sch}{\mathrm{Sch}}

\newcommand{\ShB}{\mathbf{B}}
\newcommand{\ShBr}{\mathbf{Br}}
\newcommand{\ShCat}{\mathbf{Cat}}

\newcommand{\ShFilt}{\mathbf{Filt}}
\newcommand{\ShExFilt}{\mathbf{ExFilt}}

\DeclareMathOperator{\PGL}{PGL}
\newcommand{\ShBPGL}{\mathbf{BPGL}}

\newcommand{\Gm}{\mathds{G}_{m}}

\newcommand{\SB}{\mathrm{SB}}

\DeclareMathOperator*{\colim}{colim}

\newcommand{\et}{\mathrm{\acute{e}t}}
\newcommand{\Et}{\mathrm{\acute{E}t}}

\newcommand{\fppf}{\mathrm{fppf}}

\DeclareMathOperator{\Spec}{Spec}

\newcommand{\we}{\simeq}
\newcommand{\iso}{\cong}

\theoremstyle{plain}
\newtheorem{theorem}{Theorem}[section]
\newtheorem*{theorem*}{Theorem}
\newtheorem{lemma}[theorem]{Lemma}

\newtheorem{proposition}[theorem]{Proposition}

\newtheorem{corollary}[theorem]{Corollary}

\newtheoremstyle{named}{}{}{\itshape}{}{\bfseries}{.}{.5em}{#1 \thmnote{#3}}
\theoremstyle{named}

\theoremstyle{definition}
\newtheorem{definition}[theorem]{Definition}

\newtheorem{example}[theorem]{Example}
\newtheorem{question}[theorem]{Question}

\newtheorem{remark}[theorem]{Remark}
\newtheorem{warning}[theorem]{Warning}

\begin{document}
\maketitle

\begin{abstract}
    \noindent
    We prove descent theorems for semiorthogonal decompositions using
    techniques from derived algebraic geometry. Our methods allow us to capture
    more general filtrations of derived categories and even marked filtrations,
    where one descends not only admissible subcategories but also preferred
    objects.

    \paragraph{Key Words.}
    Derived categories, semiorthogonal decompositions, twisted forms, and descent.

    \paragraph{Mathematics Subject Classification 2010.}
    \href{http://www.ams.org/mathscinet/msc/msc2010.html?t=14Fxx&btn=Current}{14F05},
    \href{http://www.ams.org/mathscinet/msc/msc2010.html?t=14Fxx&btn=Current}{14F22},
    \href{http://www.ams.org/mathscinet/msc/msc2010.html?t=14Mxx&btn=Current}{14M17},
    \href{http://www.ams.org/mathscinet/msc/msc2010.html?t=18Exx&btn=Current}{18E30}.
\end{abstract}


\section{Introduction}

The theory of semiorthogonal decompositions is a crucial tool in understanding derived
categories of schemes, especially derived categories of smooth proper
schemes over fields. This paper concerns the construction of new semiorthogonal
decompositions from known ones in various situations that involve ``twists" or
``descent". We expect our work will be useful for future explorations of the relationship
between derived categories and rationality as studied in~\cite{auel-bernardara}. We consider two questions.

\begin{question}[Absolute descent]\label{q:1}
    Suppose that $G$ is a group scheme acting on a scheme $X$ and suppose that
    there is a semiorthogonal decomposition
    $$\Perfscr(X)\we\langle\Escr_0,\ldots,\Escr_r\rangle.$$ When does
    $\Perfscr([X/G])$ admit a descended semiorthogonal decomposition?
\end{question}
 
Here $[X/G]$ is the stack quotient of $X$ by $G$. A sufficient condition on the
action, namely that it be upper triangular with respect to the semiorthogonal
decomposition, was first obtained by Elagin in \cite{elagin}.

\begin{question}[Relative descent]\label{q:2}
    Let $X$ and $Y$ be $k$-schemes for some
    field $k$ and suppose that $X$ and $Y$ become isomorphic over some separable
    extension $l/k$. If $X$
    admits a $k$-linear semiorthogonal decomposition
    $$\Perfscr(X)\we\langle\Escr_0,\ldots,\Escr_r\rangle,$$ then $\Perfscr(Y_l)$
    admits an $l$-linear semiorthogonal decomposition
    $\Perfscr(Y_l)\we\langle(\Escr_0)_l,\ldots,(\Escr_r)_l\rangle$.\footnote{In this paper, $\Perfscr(X)$
    will denote the small idempotent complete stable $\infty$-category of perfect complexes of
    $\Oscr_X$-modules on $X$. The homotopy category of $\Perfscr(X)$ is thus the usual
    triangulated category $\Perf(X)$ of perfect complexes on $X$. When $X$ is
    regular, noetherian, and quasi-separated, $\Perfscr(X)\we\Dscr^b(X)$, the
    bounded derived category of coherent sheaves on $X$,
    but in general these are different.}
    When does this semiorthogonal decomposition on $Y_l$ descend to $Y$?
\end{question}

For example, Blunk--Sierra--Smith~\cite{blunk-sierra-smith} constructed a semiorthogonal decomposition on
the blowup of $\PP^2$ at $3$ points which descends to any twisted form.
Similarly, Bernardara showed \cite{bernardara} that
Beilinson's semiorthogonal decomposition of the derived category of
$\Perfscr(\PP^{n}_k)$ descends to $\Perfscr(Y)$ if $Y$ is the Severi--Brauer
variety of a central simple algebra of degree $n+1$ over $k$.

Note that Question~\ref{q:2} can be reduced to Question~\ref{q:1} by replacing
$l$ with its Galois closure, but we will
see that there is some advantage in treating the relative case separately.
In both cases, the philosophical answer is that the
semiorthogonal decomposition should descend as long as the group action or
descent data preserves the semiorthogonal decomposition. 

One of the goals of this paper is to make this intuition precise. In
particular, we wanted to understand why it is enough to check only
$1$-categorical information in the results of Elagin~\cite{elagin},
Auel--Bernardara~\cite{auel-bernardara}, and
Ballard--Duncan--McFaddin~\cite{ballard-duncan-mcfaddin} when in
principal one has to glue higher homotopical objects, a process which requires
higher-degree analogues of the cocycle condition in general.

The starting point of our
approach is the idea that a semiorthogonal
decomposition constitutes a special kind of filtration on the derived category.
We will see that after descending the filtration admissibility
comes along for the ride.

Fix a base scheme $S$ and let $\ShCat$ be the stack\footnote{In this paper our
stacks are fppf stacks.} which assigns to each affine $\Spec R\rightarrow S$ the $\infty$-category $\ShCat(R)=\Cat_R$ of small
idempotent complete $R$-linear stable $\infty$-categories.\footnote{We will
work with stable $\infty$-categories instead of dg categories out of personal
preference. The theory could also be developed in the language of the
equivalent theory of dg categories.}
Now, fix a poset $P$ and let $\ShFilt_P$ be the prestack of linear categories
equipped with $P$-shaped filtrations.

\begin{theorem}[Filtrations, Theorem~\ref{thm:desc} and Proposition~\ref{prop:fib-disc}]\label{thm:filtrations}
    Let $P$ be a poset. The prestack $\ShFilt_P$ of $P$-shaped filtrations is a
    stack. Moreover, the forgetful functor $\ShFilt_P\rightarrow\ShCat$ has
    discrete fibers.
\end{theorem}

The proof of Theorem~\ref{thm:filtrations} is rather formal, although it has
important consequences for the questions above. The next result is
deeper. Let $P$ be a poset and let $\mathbf{Sod}_P\subseteq\ShFilt_P$ be the subprestack
of $P$-shaped semiorthogonal decompositions. 

\begin{theorem}[Semiorthogonal decompositions, Corollary~\ref{cor:sod-desc}]\label{thm:sods}
    For any poset $P$, the prestack $\mathbf{Sod}_P$ is a stack.
\end{theorem}

This theorem says that the only obstruction to descending a semiorthogonal
decomposition is descending the associated filtration. It is perhaps one of the
main insights of this paper and demonstrates the local
nature of admissibility which, on first pass, is rather surprising.
The proof of Theorem~\ref{thm:sods} highlights the power of working in a
higher categorical setting where problems of descent and base
change can be formulated and solved in an elegant manner.

While writing this paper, we were made aware of a different approach to variants of our
results by Belmans--Okawa--Ricolfi \cite{bor-sod}. These authors also studied a
version of the stack which in our notation would be denoted by
$\mathbf{Sod}^{\Perfscr(X)}_{[n]}$ where $X$ is scheme of characteristic zero. They
proved fppf descent in this setting using a completely different approach,
namely, by studying Fourier-Mukai kernels instead of appealing to machinery
derived algebraic geometry. While their descent results are more restricted,
they were able to access some deep geometric information about this stack.
Among other things they proved that these stacks are, in fact, algebraic
spaces; we conclude only that these stacks are discrete in this paper.

Theorems~\ref{thm:filtrations} and~\ref{thm:sods} provide complete
answers to Questions~\ref{q:1} and~\ref{q:2} --- the stack $\mathbf{Sod}_P$
controls the descent problems posed by both questions. See also
Sections~\ref{sec:filquo} and~\ref{sec:desctwist}.
As we illustrate in Section~\ref{sec:ex} these problems are actually quite tractable in practice. Moreover, we also give results on
descending individual objects in the pieces of semiorthogonal decompositions. 

Theorem~\ref{thm:filtrations} is proved in Section~\ref{sec:filtrations} and
Theorem~\ref{thm:sods} in Section~\ref{sec:sod}. The twisted Brauer space
perspective is developed in Section~\ref{sec:tbs} and many examples are given
in Section~\ref{sec:ex}.

Previously, Elagin studied descent for
semiorthogonal decompositions in~\cite{elagin}, showing that one could descend
semiorthogonal decompositions of triangulated categories along certain
comonads. The key condition of Elagin's main theorem is that the comonad should
be upper triangular with respect to the semiorthogonal decomposition, meaning
in other words that the comonad respects the filtration coming from the
semiorthogonal decomposition. Elagin's work was recently revisited by
Shinder who gave a new proof~\cite{shinder}. A similar approach is given in
work of Bergh and Schn\"urer~\cite{bergh-schnurer}.

In practice, there are many cases where working algebraic geometers have
established descent for semiorthogonal decompositions or exceptional
collections in interesting settings
by hand. We list some here. Historically, the first is Bernardara's work mentioned above on Severi--Brauer
schemes~\cite{bernardara} followed by work of
Blunk--Sierra--Smith~\cite{blunk-sierra-smith} on degree six del Pezzo surfaces,
unpublished work of Blunk~\cite{blunk2012derived} on some twisted
Grassmannians, and work of
Baek~\cite{baek} on twisted Grassmannians in
general. Perhaps the two most impressive works in this
direction are the paper of Auel and Bernardara on derived categories of del
Pezzo varieties over general fields~\cite{auel-bernardara} and the work of
Ballard--Duncan--McFaddin on toric varieties~\cite{ballard-duncan-mcfaddin}.
In these the authors construct explicit vector bundles generating their
semiorthogonal decompositions. For more on the connection of our work with this
previous work, see Section~\ref{sec:ex}.

Under the hood, our approach bears some similarity to the recent work of
Scherotzke--Sibilla--Talpo in~\cite{scherotzke-sibilla-talpo} who prove that
$\infty$-categories equipped with finite semiorthogonal decompositions indexed by
possibly varying index sets
admit certain limits. In our work, the indexing sets will be fixed.

The main ideas in this paper go back to 2013 when a first draft of the paper
was produced by the first author. However, at that time, Alexander Kuznetsov pointed out
Elagin's work and it was decided not to pursue the project further.
In the meantime, the problem of descent for semiorthogonal decompositions has
returned again and again and it seemed like those early results were worth
making public after all. This has been done here with many simplifications and
extensions.

\paragraph{Notation.}
Let $\Sscr$ be the $\infty$-category of spaces and let $\Sp$ be the
$\infty$-category of spectra.
If $\Cscr$ is an $\infty$-category and $x,y\in\Cscr$, then
$\Map_\Cscr(x,y)$
denotes the mapping space from $x$ to $y$ and, if $\Cscr$ is stable, then $\ShMap_\Cscr(x,y)$ denotes the
mapping spectrum.

\paragraph{Conventions.}
In this paper, a {\bf prestack} on a small $\infty$-category $\Cscr$ will mean a functor $\Cscr^{\op} \rightarrow
\widehat{\Cat}_{\infty}$ where $\widehat{\Cat}_{\infty}$ is the
$\infty$-category of possibly large $\infty$-categories. If $\tau$ is a
topology on $\Cscr$, by a {\bf $\tau$-stack} we mean a prestack which satisfies
$\tau$-descent. Functors of the form $\Cscr^{\op} \rightarrow \Sscr$ will be
called {\bf presheaves} and those that satisfy $\tau$-descent are {\bf
$\tau$-sheaves}. If $X$ is a presheaf on $\Aff_R$, the category of affine
schemes over a commutative ring $R$, there is a symmetric monoidal stable
$\infty$-category $\Perfscr(X)$ of perfect complexes on $X$ defined as
$\lim_{\Spec S\rightarrow X}\Perfscr(\Spec S)$; see Example~\ref{ex:stable}.

\paragraph{Acknowledgments.}
We would like to thank Aravind Asok, Asher Auel, Tom Bachmann, Jesse Burke, David Gepner, Rune Haugseng,
Aaron Mazel-Gee, Patrick McFaddin, Dan Murfet and Nicol\`o Sibilla for numerous
helpful conversations on the topic. We would especially like to thank Asher
Auel, Pieter Belmans and an anonymous referee for very helpful comments on an earlier draft.

BA was supported by NSF Grant DMS-1552766.
This material is based upon work supported by the National Science Foundation
under Grant No. DMS-1440140 while BA and EE were in residence at the
Mathematical Sciences Research Institute in Berkeley, California, during the
Spring 2019 semester. Part of this work was also carried out when EE was a member of the Center for Symmetry and Deformation at the University of Copenhagen.


\section{The stack of filtrations and variants}\label{sec:filtrations}

In this section we prove that the $\infty$-category of possibly marked filtered
idempotent complete stable $\infty$-categories forms an fppf sheaf.

\subsection{Background on stable $\infty$-categories}

We give a brief set of definitions and remarks
about stable $\infty$-categories. For details, see~\cite[Chapter~1]{ha}.

\begin{definition}
    \begin{enumerate}
        \item[(a)] An $\infty$-category $\Cscr$ is {\bf stable} if it is pointed,
            admits finite limits and finite colimits, and if the suspension
            functor $\Sigma\colon\Cscr\rightarrow\Cscr$ is an equivalence.
        \item[(b)] A functor $F\colon\Cscr\rightarrow\Dscr$ between two stable
            $\infty$-categories is {\bf exact} if it preserves finite limits or
            equivalently (for stable $\infty$-categories) finite colimits.
        \item[(c)] An $\infty$-category $\Cscr$ is {\bf idempotent complete} if every
            idempotent in $\Cscr$ admits a splitting.
    \end{enumerate}
\end{definition}

\begin{remark}
    If $\Cscr$ is a stable $\infty$-category, then the homotopy category
    $\Ho(\Cscr)$, which is just an ordinary category, admits a canonical
    triangulated category structure \cite[Theorem 1.1.2.14]{ha}. A functor $F$ as in (b) is exact if and
    only if the functor $\Ho(F)\colon\Ho(\Cscr)\rightarrow\Ho(\Dscr)$ is exact.
    Additionally, $\Cscr$ is idempotent complete if and only if $\Ho(\Cscr)$ is
    idempotent complete.
\end{remark}

\begin{example}\label{ex:stable}
    \begin{enumerate}
        \item[(1)] If $R$ is a commutative ring, then $\Dscr(R)$ and
            $\Perfscr(R)$ are a idempotent complete stable $\infty$-category.
            In this case, $\Perfscr(R)\subseteq\Dscr(R)$ is the full
            subcategory of compact objects, i.e., those objects $M$ such that
            the functor $\Hom_{\D(R)}(M,-)$ preserves arbitrary coproducts,
            where $\D(R)=\Ho(\Dscr(R))$.
        \item[(2)] If $S$ is an algebraic stack, then we can define
            $\Perfscr(S)$ by right Kan extension. Namely, $\Perfscr(S)$ is the
            value at $S$ of the diagonal functor in the commutative diagram
            $$\xymatrix{
                \Aff^{\op}\ar[d]\ar[r]^{\Perfscr(-)}&\Cat_{\infty}^\perf\\
                \Pscr(\Aff)^{\op}\ar@{.>}[ur]_{\R\Perfscr(-)},
            }$$
            where $\Pscr(\Aff)$ is the $\infty$-category of presheaves of
            spaces on $\Aff$. Practically speaking, we compute
            $$\Perfscr(S)=\lim_{\Spec R\rightarrow S}\Perfscr(R),$$
            although since $\Aff^{\op}$ is not small, care needs to be taken to ensure that this limit exists in
            $\Cat_{\infty}^\perf$. When $S$ is quasi-compact and
            quasi-separated with quasi-affine diagonal, for example $[X/G]$
            where $G$ is an affine algebraic group acting on a quasicompact
            quasiseparated scheme $X$, this limit does exist. Indeed, there is
            a cover $U\rightarrow S$ where $U$ is affine and where each
            $U\times_S\cdots\times_S U$ is quasicompact and quasi-affine. Then,
            $\Perfscr(S)\we\lim_\Delta\Perfscr(\check{C}(U))$, where
            $\check{C}(U)$ is the \v{C}ech complex of $U\rightarrow S$.
        \item[(3)] The $\infty$-categories $\Sp$ of spectra and $\Sp^\omega$ of
            finite spectra are stable $\infty$-categories. The former plays the
            role of $\Dscr(\SS)$ while the latter plays the role of
            $\Perfscr(\SS)$, where $\SS$ is the sphere spectrum, the initial
            commutative ring in stable homotopy theory.
        \item[(4)] If $C$ is a pretriangulated dg category, then there is a
            naturally associated stable $\infty$-category $\N_{\dg}(C)$.
    \end{enumerate}
\end{example}

The theory of idempotent complete stable $\infty$-categories and exact functors is organized into
an $\infty$-category $\Cat_{\infty}^\perf$.\footnote{For technical purposes
later, we will need to use that $\Cat_\infty^\perf$ admits a natural $(\infty,2)$-categorical
structure. Indeed, for two idempotent complete stable
$\infty$-categories there is an idempotent complete stable $\infty$-category
$\Fun^\ex(\Cscr,\Dscr)$ of exact functors. The underlying $\infty$-groupoid
(obtained by forgetting all non-invertible natural transformations between
exact functors) $\iota\Fun^\ex(\Cscr,\Dscr)$ is naturally equivalent to the
mapping space $\Map_{\Cat_\infty^\perf}(\Cscr,\Dscr)$.} Moreover, this
$\infty$-category admits a natural symmetric monoidal structure where, for
example, if $A$ and $B$ are rings, then
$\Perfscr(A)\otimes_{\Perfscr(\ZZ)}\Perfscr(B)\we\Perfscr(A\otimes_{\ZZ}^\L B)$, where
$A\otimes_{\ZZ}^\L B$ denotes the derived tensor product viewed for example as a dg algebra.

If $R$ is a commutative ring, then $\Perfscr(R)$ admits a natural symmetric
monoidal structure (which on the homotopy category gives the derived tensor
product of $R$-modules), whence we may view $\Perfscr(R)$ as a ``highly-structured" commutative ring, more precisely an $E_{\infty}$-ring, object in $\Cat_\infty^\perf$. We thus form categories of modules over them \cite[Chapter 3]{ha} and define
$$\Cat_R=\Mod_{\Perfscr(R)}(\Cat_\infty^\perf).$$
The objects of $\Cat_R$ are $R$-linear idempotent complete stable
$\infty$-categories. We will just call these {\bf $R$-linear categories} for
simplicity. Note that $\Cat_R$ is equivalent to the theory of idempotent
complete stable $\infty$-categories and exact functors enriched in
$\Dscr(R)$.

\begin{remark}
    We can also consider $R$-linear dg categories. There is a model structure
    on the category $\mathrm{dgcat}_R$ of $R$-linear dg categories where the
    weak equivalences are the Morita equivalences: a functor $C\rightarrow E$
    of dg categories is a weak equivalence if the induced functor
    $\D_{\dg}(C)\rightarrow\D_{\dg}(E)$ on dg module categories is an
    equivalence. The $\infty$-category associated to this model category is
    equivalent to $\Cat_R$. Unfortunately, no published reference for this fact
    is known, but see~\cite{cohn} for an unpublished version. An analogous
    statement, due to Shipley~\cite{shipley}, says that the homotopy theory of dg algebras agrees with the
    homotopy theory of $\EE_1$-ring spectra over $R$ if $R$ is a
    commutative ring.
\end{remark}

The assignment $R\mapsto\Cat_R$ can be given the structure of a presheaf
$$\ShCat\colon\CAlg\we\Aff^{\op}\rightarrow\widehat{\Cat}_\infty$$ of large $\infty$-categories
on the category $\Aff^{\op}$ of affine schemes (over $\ZZ$) with $$\ShCat(\Spec
R)=\ShCat(R)=\Cat_R.$$ We call this the {\bf prestack of linear categories}.
We can also work relative to a base commutative ring $R$ and restrict this to a presheaf on $\Aff_R=\CAlg_R^{\op}$, the
category of affine $R$-schemes.

Now, let $S$ be an algebraic stack.
We let $\Cat_S=\ShCat(S)$ be the value on $S$ of the right Kan extension of
$\ShCat\colon\Aff^\op\rightarrow\widehat{\Cat}_\infty$
along the inclusion $\Aff^\op\rightarrow\Pscr(\Aff)^{\op}$. We will call the
objects of $\Cat_S$ simply $\Oscr_S$-linear categories.

\begin{warning}\label{warn:1-aff}
    There is a canonical functor
    $\Mod_{\Perfscr(S)}(\Cat_\infty^\perf) \rightarrow \Cat_S$. We warn the reader
    that this functor is \emph{not} always an equivalence. Algebraic stacks for
    which this is the case are called \df{1-affine} \cite[Definition
    1.3.7]{1-affine}. Examples of $1$-affine algebraic stacks include
    \begin{enumerate}
        \item[(a)]  quasi-compact quasi-separated schemes, or more generally
        \item[(b)] quasi-compact quasi-separated algebraic spaces~\cite[Theorem
            2.1.1]{1-affine}, and even
        \item[(c)] stack quotients of quasi-compact quasi-separated algebraic spaces by affine algebraic
            groups of finite type \cite[Theorem 2.2.2]{1-affine}.
    \end{enumerate}
   For many results in this paper we require that $S$ is
    $1$-affine so that we can view an $\Oscr_S$-linear category more concretely as a small
    idempotent complete stable $\infty$-category with extra structure, namely
    with the structure of a $\Perfscr(S)$-module structure. On the other hand, defining $\Cat_S$ via right Kan extension has the advantage that it inherits descent properties from its value on affine schemes; see \cite[Theorem 1.5.7]{1-affine}.
\end{warning}

\subsection{Filtrations on stable $\infty$-categories} \label{subsec:filt}

We will work everywhere relative to a fixed poset.

\begin{definition}
    A {\bf poset} is a partially ordered set, i.e., a set $P$ together with a binary
    relation $\leq$ satisfying the reflexivity, antisymmetry, and transitivity conditions.
    When we say that an $\infty$-category is a poset, we mean that it is equivalent
    to the nerve of a poset. Equivalently, a poset is an $\infty$-category $\Cscr$
    such that for each pair of objects $x,y\in\Cscr$, the mapping space
    $\Map_\Cscr(x,y)$ is either empty or contractible. We will make no
    notational distinction between considering a poset $P$ as an ordinary
    category or as an $\infty$-category. The posets that we care about in this paper are $P$ is {\bf filtered}: every finite
    set of elements of $P$ has an upper bound.
\end{definition}

Let $P$ be a poset.

\newcommand{\Catex}{\Cat_\infty^\ex}
\newcommand{\FCatex}{\Filt\Cat_\infty^\ex}
\newcommand{\FPCatex}{\Filt_P\Cat_\infty^\ex}

\begin{example} For example, $P$ could be
    \begin{enumerate}
        \item[(i)] the totally ordered set with $n+1$ elements
            $[n]=\{0<1<\cdots<n\}$ for an integer $n\geq 0$;
        \item[(ii)] products such as $[m]\times [n]$, for integers $m,n\geq 0$,
            with the product partial order: $(i,j)\leq(l,k)$ if and only if $i\leq l$ and
            $j\leq k$;
        \item[(iii)] $\NN=\{0,1,2,\ldots\}$ or $\ZZ$ with the usual total
            orders;
        \item[(iv)] the set of finite subsets of a given set with the partial order
            given by set containment.
    \end{enumerate}
\end{example}

Let $S$ be a $1$-affine algebraic stack.
We will study $P$-shaped filtrations on $\Oscr_S$-linear categories.

\begin{definition}
    The $\infty$-category $\Filt_P\Cat_S$ of {\bf $P$-shaped filtrations of
    $\Oscr_S$-linear categories}
    is the full subcategory of the functor category $\Fun(P,\Cat_S)$ on
    those functors $\F_\star\Cscr\colon P\rightarrow\Cat_S$ such that
    for $p\leq q$ in $P$ the induced map $\F_p\Cscr\rightarrow \F_q\Cscr$
    is fully faithful.\footnote{Note that, by $1$-affineness, a functor $\Escr\rightarrow\Cscr$ of
    $\Oscr_S$-linear categories is fully faithful if and only if
    corresponding functor $\Ho(\Escr)\rightarrow\Ho(\Cscr)$ of triangulated
    homotopy categories is fully faithful.}
\end{definition}

\begin{example}
    Evaluation at $0$ gives an equivalence $\Filt_{[0]}\Cat_S\rightarrow\Cat_S$.
\end{example}

In the previous definition, there is no ambient $\Oscr_S$-linear category that is being filtered.

\begin{definition} \label{def:p-indexed-filt}
    A {\bf $P$-shaped filtration on an $\Oscr_S$-linear category $\Cscr$} is a
    $P$-shaped filtration $\F_\star\Cscr$ equipped with a functor
    $$\F_{\infty}\Cscr=\colim_P\F_\star\Cscr\rightarrow\Cscr$$
    such that for each $p\in P$ the induced functor $\F_p\Cscr\rightarrow\Cscr$ is fully
    faithful.
    If the functor $\F_\infty\Cscr\rightarrow\Cscr$ is moreover an equivalence, we say that the filtration is {\bf
    exhaustive}.     
    To give a precise definition of the $\infty$-category of $P$-shaped filtrations on
    $\Cscr$, we first define the lax pullback
    $$\Filt_P\Cat_S\overrightarrow{\times}_{\Cat_S}\Delta^0$$ as the pullback
    \begin{equation} \label{eq:pull}\begin{gathered}
    \xymatrix{
    \Filt_P\Cat_S\overrightarrow{\times}_{\Cat_S}\Delta^0\ar[rr]\ar[d]   & & 
    \left(\Cat_S\right)^{\Delta^1\ar[d]}\ar[d]^{(\partial_1,\partial_0)}\\
    \Filt_P\Cat_S\ar[rr]^{(\colim_P,\Cscr)}& &\Cat_S\times\Cat_S.}
        \end{gathered}
    \end{equation}
    Thus, $\Filt_P\Cat_S\overrightarrow{\times}_{\Cat_S}\Delta^0$ is the $\infty$-category
    consisting of pairs $(\F_\star\Cscr,\Cscr)$ of a $P$-shaped filtration $\F_\star\Cscr$,
    an $\Oscr_S$-linear category $\Cscr$, and a functor
    $\colim_P\F_\star\Cscr\rightarrow\Cscr$. We let
    $$\Filt_P^\Cscr\subseteq\Filt_P\Cat_S\overrightarrow{\times}_{\Cat_S}\Delta^0$$
    be the full subcategory where each induced functor $\F_p\Cscr\rightarrow\Cscr$ is fully faithful.
    We let
    \[\ExFilt_P^\Cscr\subseteq\Filt_P^\Cscr
    \] be the full subcategory of exhaustive $P$-shaped filtrations 
    on $\Cscr$.
\end{definition}

\begin{warning} Definition~\ref{def:p-indexed-filt} is usually only
    interesting if $P$ is filtered so that the mapping spaces in $\Cscr$ can be
    computed as a filtered colimit of mapping spaces in each $\F_p\Cscr$. In particular, a filtered colimit of fully
    faithful functors is fully faithful.
    But, the present notion also lets us use discrete sets such as $\{0,1\}$ as
    our indexing sets. A $\{0,1\}$-shaped filtration on an $\Oscr_S$-linear
    category $\Cscr$ is just the data of two full subcategories
    $\F_0\Cscr\subseteq\Cscr$ and $\F_1\Cscr\subseteq\Cscr$ with no imposed
    relation. This might be useful in some situations, so we will work in this
    generality.
\end{warning}

\begin{example}
    \begin{enumerate}
        \item[(a)] A sequence
            $\F_0\Cscr\subseteq\F_1\Cscr\subseteq\cdots\subseteq\F_n\Cscr$
             of $n+1$ full subcategories
            of $\Cscr$ 
            defines an $[n]$-shaped filtration of $\Cscr$. It is exhaustive if
            and only if the last inclusion is an equivalence $\F_n\Cscr\we\Cscr$.
        \item[(b)] The Beilinson filtration on $\Perfscr(\PP^n)$ with
            $\F_p\Perfscr(\PP^n)=\langle\Oscr(0),\ldots,\Oscr(p)\rangle$ gives an
            exhaustive $[n]$-shaped filtration of $\Perfscr(\PP^n)$.
        \item[(c)] Let $X$ be a qcqs scheme.
            Let $Z_0\subseteq Z_1\subseteq\cdots$ be an $\NN$-indexed sequence
            of closed subsets of $X$, each with quasi-compact complement. Then,
            $\Perfscr(\text{$X$ on $Z_\star$})$ defines an $\NN$-shaped
            filtration on $\Perfscr(X)$. It is exhaustive if and only if each
            generic point of $X$ is contained in $Z_p$ for some finite $p$.
    \end{enumerate}
\end{example}

As we will see, the following simple lemma turns out to be the secret sauce.

\begin{lemma}\label{lem:0category}
    Let $S$ be a $1$-affine algebraic stack and let $\Cscr$ be an $\Oscr_S$-linear category.
    For any poset $P$, the $\infty$-category $\Filt_P^\Cscr$ is a
    poset. In particular, $\ExFilt_P^\Cscr$ is a poset as well.
\end{lemma}

\begin{proof}
    Let $\mathrm{Sub}^\Cscr\subseteq\left(\Cat_R\right)_{/\Cscr}$ be the full
    subcategory on the fully faithful inclusions. Then,
    $$\Filt_P^{\Cscr}\we\Fun(P,\mathrm{Sub}^\Cscr).$$
    Since the $\infty$-category of functors from one poset to another forms a
    poset, it now suffices to see
    that $\mathrm{Sub}^\Cscr$ is a poset. Let $\Dscr_0$ and $\Dscr_1$ be
    two full subcategories of $\Cscr$. There is a fiber sequence
    $$\Map_{\mathrm{Sub}^\Cscr}(\Dscr_0,\Dscr_1)\rightarrow\Map_{\Cat_R}(\Dscr_0,\Dscr_1)\rightarrow\Map_{\Cat_R}(\Dscr_0,\Cscr)$$
    of spaces,
    where the left hand term is the fiber over the fixed inclusion
    $\Dscr_0\hookrightarrow\Cscr$. Since $\Dscr_1\rightarrow\Cscr$ is fully
    faithful, the right map is an inclusion of connected components. Thus, the
    fibers are either empty or contractible. Therefore, $\mathrm{Sub}^\Cscr$ is a poset. Hence,
    $\Fun(P,\mathrm{Sub}^\Cscr)$ is a poset, which is what
    we wanted to prove.
\end{proof}

\subsection{Marked filtrations}

Marked filtrations are filtrations in which we also specify an object from each
layer in the filtration.

\begin{definition}
    Let $S$ be a $1$-affine algebraic stack.
Suppose that $\F_\star\Cscr\rightarrow\Cscr$ is a filtration on an
    $\Oscr_S$-linear category $\Cscr$. A \df{marking} of $\F_\star\Cscr\rightarrow\Cscr$ is
    the choice of an object $M_p\in\F_p\Cscr$ for each $p\in P$. The pair
    $(\F_\star\Cscr\rightarrow\Cscr,M_\ast)$ will be called a {\bf marked $P$-shaped
    filtration on $\Cscr$}. A marked $P$-shaped filtration on $\Cscr$ is called {\bf exhaustive} if the underlying filtration is exhaustive.

    We will denote by $\M\Filt_P^\Cscr$ (resp. $\M\ExFilt_P^\Cscr$) the $\infty$-category of marked (resp. marked exhaustive) filtrations on $\Cscr$.
\end{definition}

In contrast to Lemma~\ref{lem:0category}, neither the $\infty$-category
$\M\ExFilt_P^\Cscr$ nor $\M\Filt_P^\Cscr$ is a poset in general. For example, there is a
natural equivalence $\M\ExFilt_{[0]}^\Cscr\rightarrow\iota\Cscr$, where
$\iota\Cscr$ is the space of objects in $\Cscr$, given by taking
the marked object.

\begin{remark}
    There are many variants one can consider, for example by marking only
    certain parts of the filtration, or by marking the quotients
    $\frac{\F_q\Cscr}{\F_p\Cscr}$ for $p\leq q$. We leave it to the reader to
    spell out the theory in these cases.
\end{remark}

\subsection{Descent for the stack of filtrations} 

Central to our results is following theorem which is essentially due to Jacob
Lurie. We include a proof, which relies heavily on~\cite{sag}, because Lurie works in a slightly different setting.

\begin{theorem}\label{thm:lurie}
    The functor 
    \[
    \ShCat\colon\Aff^{\op} \rightarrow \widehat{\Cat}_{\infty}
    \] satisfies fppf descent.
\end{theorem}

\begin{proof} To begin with, for any commutative ring $R$, we have inclusions of subcategories
\[
\Cat_R \subseteq \mathrm{LinCat}^{\mathrm{cg}}_R \subseteq \mathrm{LinCat}^{\mathrm{st}}_R
\]
where $\mathrm{LinCat}^{\mathrm{st}}_R= \Mod_{\Modscr(R)}(\PrL)$ \cite[Variant
    D.1.5.1]{sag} and  $\mathrm{LinCat}^{\mathrm{cg}}_R$ is the full
    subcategory of those objects which are furthermore compactly generated \cite[5.5.7]{htt}. To explain
    the first inclusion, note that the functor of taking ind-objects
\[
\Ind: \Cat_R \rightarrow  \mathrm{LinCat}^{\mathrm{st}}_R
\]
factors through $\mathrm{LinCat}^{\mathrm{cg}}_R$ and identifies as the subcategory of $\mathrm{LinCat}^{\mathrm{st}}_R$ where the objects are compactly generated but the functors are those which additionally preserve compact objects (this holds by the $R$-linear version of \cite[Proposition 5.5.7.10]{htt}).

Now, \cite[Theorem D.3.6.2]{sag} implies that the functor $R \mapsto \mathrm{LinCat}^{\mathrm{st}}_R$ is an fppf sheaf. Indeed, \emph{loc. cit.} proves that it is a sheaf with respect to the universal descent topology which is finer than the fppf topology by \cite[Proposition D.3.3.1]{sag} (the cardinality assumption is trivially satisfied by morphisms of finite presentation). It then suffices to prove that
\begin{enumerate}
\item[\rm (a)] $R \mapsto \mathrm{LinCat}^{\mathrm{cg}}_R$ is an fppf sheaf,
\item[\rm (b)] for any ring $R$ and a colimit-preserving functor $F: \Cscr \rightarrow \Dscr$ in $\mathrm{LinCat}^{\mathrm{cg}}_R$ if $R \rightarrow S$ is a faithfully flat morphism of finite presentation, then if $F \otimes_{R} S$ preserves compact objects, $F$ does as well.
\end{enumerate} 
Indeed, we claim that the proof of \cite[Theorem D.5.3.1.b]{sag} proves (a).
    Since $R \mapsto \mathrm{LinCat}^{\mathrm{st}}_R$ is an fppf sheaf, it
    suffices to prove that the property of being compactly generated is local for the fppf topology. To this end, for each commutative ring $R$ and $\Cscr \in \mathrm{LinCat}^{\mathrm{cg}}_R$, define the presheaf (on $\Aff_R$)
\[
\chi_M(R') =\begin{cases}
* & \text{if $\Cscr \otimes_{S}R'$ is compactly generated} \\
\emptyset & \text{otherwise}.
\end{cases}
\]
    The first half of the proof of~\cite[Theorem~D.5.3.1(b)]{sag} shows that $\chi_M$ is an Nisnevich sheaf on the small Nisnevich site\footnote{By the small Nisnevich site we mean the category of \'etale $R$-schemes $\Et_R$ equipped with the Nisnevich topology.} for each $R'$, hence it is a Nisnevich sheaf on $\Aff_R$\footnote{One way to see this is that to be a Nisnevich sheaf, it suffices to check an excision condition involving objects on the small Nisnevich site; see \cite[Appendix A]{bachmann-hoyois} for the most general statement.}. 

Now (beginning at the end of page 2153), Lurie claims that it is a sheaf for the
finite \'etale topology. However the argument proves that it is in fact a
sheaf for the finite flat topology since, in the notation of the proof, one
only needs that $B$ is finitely generated and projective as an $A$-module.
But now, finite flat descent and Nisnevich descent implies fppf descent by
\cite[Tag 05WM]{stacks} (see also~\cite[Corollaire~17.16.2]{ega44}).

Now, we prove (b). Suppose that $M \in \Cscr$. Define the presheaf (on $\Aff_S$)
\[
\chi_M(R') =\begin{cases}
* & \text{if $F(M) \otimes_{S}R'$ is compact and} \\
\emptyset & \text{otherwise}.
\end{cases}
\]
    This makes sense as $\Dscr\rightarrow\Dscr\otimes_{R'}R''$ preserves
    compact objects for $R'\rightarrow R''$ a map of commutative $R$-algebras.
Arguing as above, one proves that $\chi_M$ is an fppf sheaf. Now, the functor
    $\Cscr \rightarrow \Cscr \otimes_R S$ preserves compact objects, hence the
    composite $\Cscr \rightarrow \Dscr \otimes_R S$ preserves compact objects,
    whence $\chi_M(S) \simeq *$ and thus $\chi_M(R) = *$ by fppf descent. This
    proves (b) and hence the theorem.
\end{proof}

From this, we will prove that various prestacks classifying filtrations are
actually fppf stacks. Let us fix a quasicompact quasiseparated $1$-affine
algebraic stack $S$ and an $\Oscr_S$-linear category $\Cscr$. We consider the prestacks
$$\Aff_S^\op\rightarrow\widehat{\Cat}_\infty$$ given by
\begin{enumerate}
    \item[(a)] $\mathbf{Filt}_P\colon \Spec R \mapsto \Filt_P(\Cat_R)$,
    \item[(b)] $\mathbf{Filt}^{\Cscr}_P\colon \Spec R \mapsto \Filt^{\Cscr\otimes_{\Perfscr(S)}\Perfscr(R)}_P(\Cat_R)$,
    \item[(c)] $\mathbf{ExFilt}_P^{\Cscr}\colon R \mapsto \Ex\Filt_P^{\Cscr\otimes_{\Perfscr(S)}\Perfscr(R)}(\Cat_R)$,
\end{enumerate}
as well as their marked variants $\mathbf{MFilt}_P, \mathbf{MFilt}^{\Cscr}_P$ and $\mathbf{MExFilt}_P^\Cscr$.

\begin{theorem} \label{thm:desc} Let $P$ be a poset.
    \begin{enumerate}
        \item[{\rm (i)}] The prestack $\mathbf{Filt}_P$ satisfies fppf
            descent. Consequently, $\mathbf{Filt}^{\Cscr}_P$ and $\mathbf{ExFilt}_P^{\Cscr}$ also satisfy fppf descent.
        \item[{\rm (ii)}] The prestack $\mathbf{MFilt}_P$ satisfies fppf descent. Consequently, $\mathbf{MFilt}^{\Cscr}_P$ and $\mathbf{MExFilt}_P^\Cscr$ also satisfy fppf descent.
    \end{enumerate}
\end{theorem}

\begin{proof}
    To see part (i), note that the prestack $\Spec R\mapsto\Fun(P,\mathbf{Cat}(R))$
    is an fppf stack since it is the mapping prestack $\Fun(P,\mathbf{Cat})$,
    and because $\mathbf{Cat}$ is an fppf stack by Theorem~\ref{thm:lurie}.
    To prove that $\mathbf{Filt}_P\subseteq\Fun(P,\mathbf{Cat})$ is an fppf stack, it
    is enough to show that if $\F_\star\Cscr\colon P\rightarrow\Cat_R$ is a diagram which
    becomes a diagram of fully faithful maps after base changing along a
    faithfully flat map $S\rightarrow T$, then each
    $\F_p\Cscr\rightarrow\F_q\Cscr$ is already fully faithful. But, this follows
    from the fact that mapping spaces themselves may be calculated flat
    locally. (See for example~\cite[Corollary~6.11]{dag11}.) Similarly,
    the prestack $\ShCat^{\Delta^1}\colon\Spec R\mapsto\Cat_R^{\Delta^1}$ is
    $\Fun(\Delta^1,\ShCat)$ and so defines an fppf stack.
    Since stacks are
    stable under pullbacks in prestacks, we get that
    $\ShFilt_P\times_{\ShCat\times\ShCat}\ShCat^{\Delta^1}\we\ShFilt_p\overrightarrow{\times}_{\ShCat}\Delta^0$
    is an fppf stack. Now,
    $\ShFilt_P^\Cscr\subseteq\ShFilt_P\times_{\ShCat\times\ShCat}\ShCat^{\Delta^1}$
    is defined by the condition that each $\F_p\Cscr\rightarrow\Cscr$ is fully
    faithful. This is fppf-local, so the fact that $\ShFilt_P^\Cscr$ is an fppf
    stack follows. The fact that $\mathbf{ExFilt}^{\Cscr}_P$ is an
    fppf stack again follows from the fact that objects in $\Cat_S$ satisfy
    fppf descent by\cite[Corollary
    6.11]{dag11}. Specifically, if $\colim_P\F_p\Cscr\rightarrow\Cscr$ is
    fppf locally an equivalence, then it is an equivalence. Indeed, set
    $\Dscr=\colim_P\F_p\Cscr$ and fix a faithfully flat map $R\rightarrow
    S$. Then,
    $\Dscr\we\lim_\Delta\Dscr\otimes_{\Perfscr(R)}\Perfscr(S^{\otimes_R\bullet+1})$
    and
    $\Cscr\we\lim_\Delta\Cscr\otimes_{\Perfscr(R)}\Perfscr(S^{\otimes_R\bullet+1})$.
    The natural transformation
    $\Dscr\otimes_{\Perfscr(R)}\Perfscr(S^{\otimes_R\bullet+1})\rightarrow\Cscr\otimes_{\Perfscr(R)}\Perfscr(S^{\otimes_R\bullet+1})$
    is a degree-wise equivalence by hypothesis, thus it is an equivalence in
    the limit.

    The prestack $\mathbf{MFilt}_P$ is computed as  pullback of prestacks
    \[
\xymatrix{
   \mathbf{MFilt}_P \ar[d] \ar[r] &  \mathbf{Filt}_P \ar[d] \\
\prod_P \mathbf{Cat}_* \ar[r]^{\prod u} & \prod_P \mathbf{Cat}.
}
\]
    Here, the right vertical arrow is induced by the inclusion of the vertices
    of $P$ into $P$ and the bottom horizontal arrow forgets the base point.
    Since stacks are closed under pullbacks and products in prestacks, we see
    that $\mathbf{MFilt}_P$ is an fppf stack. The fact that
    $\mathbf{MFilt}^{\Cscr}_P$ and $\mathbf{MExFilt}_P^\Cscr$ are fppf
    stacks follow by the same argument as for their unmarked counterparts.
\end{proof}

\begin{remark} \label{rem:shv} For any of the prestacks $\mathbf{F}$ appearing
    in Theorem~\ref{thm:desc}, we obtain presheaves by taking maximal
    subgroupoids. We will decorate the resulting presheaves by
    $\iota\mathbf{F}$. Then Theorem~\ref{thm:desc} tells us that $\iota\mathbf{F}$ are fppf sheaves since the formation of maximal subgroupoids preserves limits.
\end{remark}

By construction, there is canonical morphism of prestacks
$u\colon\mathbf{Filt}_P \rightarrow \mathbf{Cat}$  given by taking the colimit.

\begin{proposition} \label{prop:fib-disc} The fibers of
    $u\colon\ShFilt_P\rightarrow\ShCat$ are posets.
\end{proposition}

\begin{proof} 
    This is an immediate consequence of Lemma~\ref{lem:0category} since the fiber
    over $\Cscr$ is precisely $\ShExFilt_P^\Cscr$.
\end{proof}

\begin{corollary}\label{cor:0-tr} The sheaves of spaces
    $\iota\mathbf{Filt}^{\Cscr}_P$ and $\iota\mathbf{ExFilt}^{\Cscr}_P$ are $0$-truncated.
\end{corollary}

In other words, these are sheaves of sets.

\subsection{Filtrations on quotient stacks}\label{sec:filquo}

As an application of our methods, we will prove the following result which is a
version of a theorem of Elagin in \cite{elagin}.

\begin{theorem} \label{thm:main2-g} Let $P$ be a poset, $S$ a
    qcqs scheme. Let $G$ be a flat affine algebraic $S$-group scheme of finite presentation
    and let $X$ be a qcqs $S$-scheme with an action of $G$.
    Let $\F_{\star}\Perfscr(X) \rightarrow \Perfscr(X)$ be a
    $P$-filtration. If $G$ preserves the filtration on $\Perfscr(X)$, then
    there is an induced filtration
    $\F_{\star}\Perfscr([X/G]) \rightarrow \Perfscr([X/G])$.
\end{theorem}

\begin{example}\label{ex:p1gm}
    Let $S$ be a qcqs scheme and let $\mathbb{G}_{m,S}$ act on
    $\mathbb{P}_S^1$ with some weight; we will suppress the base $S$ for this
    discussion. Consider the $[1]$-shaped Beilinson filtration defined on
    $\Perfscr(\PP^1)$ via 
    \[
    0 \mapsto \F_0\Perfscr(\PP^1)=\langle\Oscr\rangle \qquad 1 \mapsto \F_1\Perfscr(\PP^1)=\langle\Oscr, \Oscr(1)\rangle;
    \]
     this filtration will be discussed more extensively in
     Section~\ref{sect:beilber}. Since any automorphism of $\mathbb{P}^1$
     preserves $\Oscr(1)$, Theorem~\ref{thm:main2-g} implies that there is a
     $[1]$-shaped filtration on perfect complexes over the stack
     $[\mathbb{P}^1/\mathbb{G}_{m}]$ (concretely, $\mathbb{G}_m$-equivariant
     complexes on $\mathbb{P}^1$) which is compatible with the quotient map
     $p\colon\mathbb{P}^1 \rightarrow [\mathbb{P}^1/\mathbb{G}_{m}]$ in the sense
     that the filtration $\F_i\Perfscr([\mathbb{P}^1/\mathbb{G}_{m}])$ pulls
     back under $p^*$ to $\F_i\Perfscr(\PP^1)$. Combined with
     Corollary~\ref{cor:main-sod-version} we see that the semiorthogonal
     decomposition on $\Perfscr(\mathbb{P}^1)$ constructed by Beilinson
     induces a semiorthogonal decomposition of perfect complexes on the quotient
     stack. Specifically, the graded pieces of the filtration above are each
     given by $$\mathrm{gr}_i\Perfscr([\PP^1/\Gm])\we\Perfscr(\ShB\Gm),$$
     which is the derived category of perfect graded complexes, for $i=0,1$.
\end{example}

\newcommand{\U}{\mathrm{U}}

\begin{remark}
    Recall that a pure complex Hodge structure of weight $i$ is a graded finite
    dimensional $\CC$-vector space.\footnote{Indeed a pure complex Hodge
    structure is a bigraded vector space $V^{*,*}$ and its weight $i$ part is the subspace given by
    those $V^{p,q}$ such that $p+q =i$. Up to reindexing, a pure weight $i$
    complex Hodge structure is thus a graded finite dimensional $\CC$-vector
    space. See~\cite[Section~2.1]{deligne-hodge2}.} If we specialize Example~\ref{ex:p1gm} to the case where
    $S=\Spec\CC$, then
    $\mathrm{gr}_i\Perfscr([\PP^1/\Gm])$ is naturally the bounded derived
    $\infty$-category of pure complex Hodge structures of weight $i$; in
    particular, there is a $t$-structure whose heart is the abelian category of
    pure complex Hodge structures of weight $i$. In order to get pure complex Hodge structures of weight $i \not= 0, 1$, one can simply twist Example~\ref{ex:p1gm} and consider the $[1]$-shaped filtration given by     \[
    0 \mapsto \F_0\Perfscr(\PP^1)=\langle\Oscr(i)\rangle \qquad 1 \mapsto \F_1\Perfscr(\PP^1)=\langle\Oscr(i), \Oscr(i+1)\rangle.
    \]
    One can make a similar
    construction in the more interesting case where $S=\Spec\RR$,
    $X\subseteq\PP^2_\RR$ is the conic curve $$x^2+y^2+z^2=0$$ (which has no
    $\RR$-points). In this case, $X$ is a twisted form of $\PP^1_\RR$ and there
    is an action of an algebraic group $\U(1)$ acting on $\PP^1_\RR$, where
    $\U(1)$ is a non-split form of $\Gm$ with $\RR$-points
    $\U(1)(\RR)\subseteq\U(1)(\CC)\iso\CC^\times$ given by
    $S^1\subseteq\CC^\times$. Thus, $[X/\U(1)]$ is an Artin stack over
    $\Spec\RR$ and it is a twisted form of $[\PP^1/\Gm]$.
    There is a $[1]$-shaped Beilinson filtration on $\Perfscr(X)$ which is a
    twisted form of the Beilinson filtration on $\Perfscr(\PP^1)$. The graded
    pieces are given by $\mathrm{gr}_0\Perfscr(X)\we\Perfscr(\RR)$ and
    $\mathrm{gr}_1\Perfscr(X)\we\Perfscr(\mathds{H})$, where $\mathds{H}$ is the quaternion
    algebra over $\RR$. Again, $\U(1)$ preserves the filtration and 
    Theorem~\ref{thm:main2-g} on the quotient stack we get a $[1]$-shaped
    filtration, which will give rise to a semiorthogonal decomposition by
    Corollary~\ref{cor:main-sod-version}, with graded piece
    $$\mathrm{gr}_i\Perfscr([X/\U(1)])$$ given by the bounded derived
    $\infty$-category of pure real Hodge structures of weight $i$.
    In particular, there is a bounded $t$-structure on each graded piece with
    heart the abelian category of pure real Hodge structures of weight $i$. To
    obtain the picture for $i \not= 0,1$ one takes Serre twists as above. For
    details on this perspective on Hodge theory, see the work of Simpson, for
    example~\cite[Lemma~19]{simpson-icm}.
\end{remark}

Thanks to Corollary~\ref{cor:0-tr}, we will see that the
Theorem~\ref{thm:main2-g} is a consequence of basic covering space theory.
Let $\ShB G$ be the classifying stack $[\Spec S/G]$. It is $1$-affine by
Warning~\ref{warn:1-aff}.
Since $\ShCat$ is an fppf stack by Theorem~\ref{thm:lurie}, there is a morphism of stacks\footnote{Here, we use that $\ShB G$ is the quotient stack in the fppf topology, following the conventions of \cite[Tag 044O]{stacks}.}
$$\ShB G \rightarrow \ShCat$$
classifying the $G$-action on $\Perfscr(X)$.
Now, suppose that we have a $P$-filtration $\F_{\star}\Perfscr(X) \rightarrow \Perfscr(X)$. This is equivalent to giving a commutative diagram of stacks
\begin{equation} \label{lift}
\begin{tikzcd}
* \ar{d} \ar{r} & \mathbf{Filt}_P \ar{d} \\ 
\ShB G \ar{r}  & \mathbf{Cat}.
\end{tikzcd}
\end{equation}

\begin{lemma} \label{prop:g-lift}
    A lift $\ShB G \rightarrow \mathbf{Filt}_P$ filling in the
    diagram~\eqref{lift} exists if and only if a lift in 
\begin{equation} \label{lift-0}
\begin{tikzcd}
 & \pi_0\mathbf{Aut}_{\mathbf{Filt}_P}(\F_{\star}\Perfscr(X) \rightarrow \Perfscr(X)) \ar{d} \\ 
G \ar{r} \ar[dashed]{ur} & \pi_0\mathbf{Aut}_{\mathbf{Cat}}(\Perf(X))
\end{tikzcd}
\end{equation}
exists in the category of fppf sheaves of groups on $\Aff_S$.
\end{lemma} 

\begin{proof} The ``only if" direction is clear. Suppose that the filler
    in~\eqref{lift-0} exists. We may assume that $S=X$. The horizontal arrows
    in~\eqref{lift} factor through the underlying sheaves of spaces so we
    obtain a commutative diagram
\begin{equation} \label{lift-1}
\begin{tikzcd}
* \ar{d} \ar{r} & \iota\mathbf{Filt}_{P, \F_{\star}\Perfscr(X) \rightarrow \Perfscr(X)} \ar{d}{p} \\ 
\ShB G \ar{r} & \iota\mathbf{Cat}_{\Perfscr(X)},
\end{tikzcd}
\end{equation}
where $\iota\mathbf{Cat}_{\Perfscr(X)}$ (resp. $\iota\mathbf{Filt}_{P,
    \F_{\star}\Perfscr(X) \rightarrow \Perfscr(X)}$) denotes the connected
    component of $\iota\ShCat$ (resp. $\iota\ShFilt_P$) corresponding to the base point $\Perfscr(X)$ (resp.
    $\F_{\star}\Perfscr(X) \rightarrow \Perfscr(X))$ so it suffices to solve
    this lifting problem. By Corollary~\ref{cor:0-tr}, the right hand vertical
    map is $0$-truncated and so covering space theory in the $\infty$-topos of
    fppf sheaves on $\Aff_S$ tells us that the
    existence of a filler is assured if the map $G \rightarrow
    \pi_1(\iota\mathbf{Cat}_{\Perfscr(X)}, \Perfscr(X))$ factors through 
\[
p_*(\pi_1(\iota\mathbf{Filt}_{P, \F_{\star}\Perfscr(X) \rightarrow \Perfscr(X)},  \F_{\star}\Perfscr(X) \rightarrow \Perfscr(X))) \subset  \pi_1(\iota\mathbf{Cat}_{\Perfscr(X)}, \Perfscr(X)).
\]
This is exactly the existence of a lift as in~\eqref{lift-0}.
\end{proof}

\begin{proof} [Proof of Theorem~\ref{thm:main2-g}] 
    The statement of Theorem~\ref{thm:main2-g} asserts the existence of a
    filtration on $\Perfscr([X/G])$. The assumption guarantees by
    Lemma~\ref{prop:g-lift} that
    there is a point in $\mathbf{Filt}_P(\ShB G)$ lying over
    $\Perfscr(X)\in\ShCat(\ShB G)$. In other words,
    $\F_\star\Perfscr(X)\rightarrow\Perfscr(X)$ admits the structure of a $G$-equivariant
    filtration. Now, applying homotopy $G$-fixed points, we obtain a filtration
    $(\F_\star\Perfscr(X))^{hG}\rightarrow\Perfscr(X)^{hG}\we\Perfscr([X/G])$,
    as desired.
\end{proof}

\begin{remark} \label{remark:triangulated} Theorem~\ref{thm:main2-g} is somewhat surprising since
    \emph{a priori} it involves manipulating higher-categorical objects which
    usually involves an infinite list of coherent descent data. The explanation
    that this is not necessary for filtrations is given by
    Lemma~\ref{lem:0category}. This also explains why Elagin in \cite{elagin}
    could stay within the realm of triangulated categories which usually does
    not interact well with descent problems.
\end{remark}


\section{Semiorthogonal decompositions}\label{sec:sod}

The previous sections dealt with general filtrations. Now, we deal with
semiorthogonal decompositions in the sense of~\cite{bondal-kapranov}. In particular, we prove Theorem~\ref{thm:sods} via
Theorem~\ref{thm:et-sod},
which says that a subcategory $\Ascr\subseteq\Cscr$ is admissible if and only
if it is fppf-locally admissible.

\subsection{Admissibility}

Let $S$ be a $1$-affine algebraic stack. We review in this section the definitions and standard facts about
admissible subcategories.

\begin{definition}[\cite{bondal-kapranov}]
    Let $\Ascr\subseteq\Cscr$ be a fully faithful inclusion of
    $\Oscr_S$-linear categories. We say that $\Ascr$ is {\bf right-admissible}
    in $\Cscr$
    if the inclusion admits an $\Oscr_S$-linear right adjoint. Similarly, $\Ascr$ is {\bf
    left-admissible} in $\Cscr$ if the inclusion admits an $\Oscr_S$-linear left adjoint.
    If the inclusion admits both adjoints, we say that $\Ascr\subseteq\Cscr$ is
    {\bf admissible}.
\end{definition}

\begin{remark}
    In the case of greatest interest, $\Cscr$ will be dualizable (i.e., smooth and proper) as an $\Oscr_S$-linear
    category in which case
    the three notions of admissibility for a full subcategory
    $\Ascr\subseteq\Cscr$ agree, and are all furthermore equivalent to the
    smoothness of $\Ascr$.
\end{remark}

\begin{definition}
    If $\Ascr \subseteq \Cscr$ is an inclusion of stable $\infty$-categories, then
    the left (resp. right) orthogonal of $\Ascr$, denoted by $^{\perp}\Ascr$ (resp.
    $\Ascr^{\perp}$) is the full subcategory of $\Cscr$ spanned by objects $y \in
    \Cscr$ such that $\MapSp_{\Cscr}(y, x)$ (resp. $\MapSp_{\Cscr}(x, y)$) is
    contractible for all $x \in \Ascr$. If the ambient stable $\infty$-category
    is ambiguous, we will write
    $(^{\perp}\Ascr)_{\Cscr}$ (resp. $(\Ascr^{\perp})_{\Cscr}$) to avoid confusion.
\end{definition}

The next well-known proposition furnishes a list of checkable criteria for right-admissibility. 

\begin{proposition}[\cite{bondal-kapranov}]\label{prop:all-equal}
    Let $S$ be a $1$-affine algebraic stack.
    Suppose that $i\colon\Ascr\subseteq\Cscr$ is a fully faithful $\Oscr_S$-linear functor of
    $\Oscr_S$-linear categories. Then, the following conditions
    on $i$ are equivalent.
    \begin{enumerate}
        \item[{\em (1)}]   For every $x$ in $\Cscr$ there is a cofiber sequence $y\rightarrow
            x\rightarrow z$ where $y\in\Ascr$ and $z\in\Ascr^\perp$.
        \item[{\em (2)}]  There is a $t$-structure $(\Cscr_{\geq 0}, \Cscr_{\leq 0})$ on $\Cscr$ for which $\Cscr_{\geq 0} \simeq \Ascr$.
        \item[{\em (3)}]   The functor $i$ admits a right adjoint.
        \item[{\em (4)}]   The inclusion $i'\colon\Ascr^\perp \subseteq \Cscr$ admits a left adjoint.
        \item[{\em (5)}]   The composition $\Ascr^\perp\stackrel{i'}{\rightarrow} \Cscr \rightarrow \Cscr/\Ascr$ is an equivalence.
    \end{enumerate}
    Furthermore, the adjoints
    appearing in $(3)$ and $(4)$ are automatically $\Oscr_S$-linear.
\end{proposition}

\begin{proof}
    One reference for most of the implications
    is~\cite[Section~1]{bondal-kapranov}, but we will sketch the arguments for
    completeness.
    The implication $(1) \Rightarrow (2)$ is immediate from the definition of a
    $t$-structure~\cite[D\'efinition~1.3.1]{bbd}.\footnote{Note that we are
    working with homological instead of cohomological indexing.} The
    implication $(2) \Rightarrow (3)$ is given by \cite[Proposition
    1.3.3]{bbd}. Let us prove that $(3)$ and $(4)$ are equivalent. Since $i\colon\Ascr
    \hookrightarrow \Cscr$ has a right adjoint $R$, we can cook up an
    endofunctor of $\Cscr$ by the formula
    \[
    L\colon X \mapsto \cofib(iRX \rightarrow X)
    \]
    given by taking the counit of the adjunction.
    The functor $L$ takes $X$ to $\Ascr^{\perp}$ since, for any $Y \in \Ascr$ we have a cofiber sequence
    \[
    \MapSp_\Cscr(iY, iRX) \rightarrow \MapSp_\Cscr(iY, X) \rightarrow
    \MapSp_\Cscr(iY, LX)
    \]
    where the first arrow is an equivalence, hence the last term is
    contractible. The check that $L$ is indeed the left adjoint is standard.
    Conversely, if $L$ is a left adjoint to the inclusion $i'\colon\Ascr^{\perp}
    \rightarrow \Cscr$, then we define the right adjoint as
    \[
    R\colon X \mapsto \fib(X \rightarrow i'LX).
    \]
    A similar argument shows that $R$ is right adjoint to $i$.

    Now we prove that $(4)$ implies $(5)$. Let $i_0$ denote the composition
    $i_0\colon\Ascr^{\perp} \rightarrow \Cscr \rightarrow \Cscr/\Ascr$ and observe that the
    adjoint
    $L\colon  \Cscr\rightarrow \Ascr^{\perp}$ of $(4)$ vanishes on $\Ascr$ and hence factors
    through $\Cscr/\Ascr$ to define a functor $L_0\colon \Cscr/\Ascr \rightarrow \Ascr^{\perp}$
    \cite[Theorem 1.3.3(i)]{nikolaus-scholze} which can be checked to be the left adjoint to
    $i_0$, and fits into the diagram
    \begin{equation} \label{eq:compatible}
    \begin{tikzcd}
    \Cscr \ar{r}{L} \ar{d} & \Ascr^{\perp}\\
    \Cscr/\Ascr \ar[swap]{ur}{L_0} &. 
    \end{tikzcd}
    \end{equation}Since $L$ and the projection $\Cscr \rightarrow \Cscr/\Ascr$ are
    essentially surjective, it is easy to see that $L_0$ is as well. It remains to show that
    $L_0$ is fully faithful. Suppose that $X, Y \in \Cscr$ with images $\overline{X},
    \overline{Y} \in \Cscr/\Ascr$. Since $(4)$ implies $(3)$, the existence of a right
    adjoint to $i$ tells us that the filtered category $\Ascr_{/Y}$ admits a final object,
    namely, $RY$. From this we compute, using the formula for mapping space in Verdier
    quotients \cite[Theorem 1.3.3(ii)]{nikolaus-scholze}
       \begin{eqnarray*}
           \MapSp_{\Cscr/\Ascr}(\overline{X}, \overline{Y}) &  \simeq &
           \colim_{Z \in \Ascr_{/LY}} \MapSp_\Cscr(X,  \cofib(Z \rightarrow Y))\\
    & \simeq & \MapSp_\Cscr(X, \cofib(iRY \rightarrow Y))\\
    & \simeq & \MapSp_\Cscr(X, LY)\\
           & \simeq & \MapSp_{\Cscr/\Ascr}(\overline{X}, \overline{LY}).
    \end{eqnarray*}
    This proves that (4) implies (5). Now (1) follows from (5) by the standard
    machinery of Bousfield localization.

     The fact that $\Perfscr(R)$ is rigid symmetric
    monoidal, meaning that every object is dualizable, implies that the
    adjoints above, if they exist, are automatically $R$-linear;
    see for example~\cite[Proposition~4.9(3)]{hss1}.
\end{proof}

\begin{example}
    The requirement that an inclusion be right or left-admissible is very
    strong. Let $k$ be a field and let $\PP^1=\PP^1_k$. If $H$ is
    a hyperplane in $\PP^1$ with complement $U$, then $\Oscr_H$ is naturally an
    object of $\Perfscr(\PP^1)$. Let $\langle H\rangle$ denote the thick
    $k$-linear
    subcategory of $\Perfscr(\PP^1)$ generated
    by $\Oscr_H$. The quotient of $\Perfscr(\PP^1)$ by $\langle H\rangle$ is
    $\Perfscr(U)$. However, it is clear that
    there can be no fully faithful functor $\Perfscr(U)\rightarrow\Perfscr(\PP^1)$
    because the mapping spectra in $\Perfscr(\PP^1)$ are perfect complexes.
\end{example}

\subsection{Semiorthogonal decompositions}

These were first introduced by Bondal~\cite{bondal} and
Bondal--Kapranov~\cite{bondal-kapranov}. We would like to define semiorthogonal
decompositions which are indexed not just by $\Delta^n$ or $\ZZ$ but also by a poset $P$. The next definition is a naive generalization of the definition that appears in \cite[Definition 4.1]{bondal-kapranov}:

\begin{definition}\label{def:adm-fil} Let $\Cscr$ be an $\Oscr_S$-linear category,
    let $P$ be a poset, and consider a $P$-shaped filtration $\F_\star\Cscr
    \rightarrow \Cscr$. We say that the filtration is {\bf admissible} (resp.
    {\bf right-admissible, left-admissible}) if for every arrow $p \rightarrow q$ in
    $P$ the fully faithful embedding $\F_p\Cscr \hookrightarrow \F_q\Cscr$ is
    admissible (resp. right-admissible, left-admissible). 
    
    We say that an (right, left) admissible
    filtration $\F_\star\Cscr\rightarrow\Cscr$ is a {\bf $P$-shaped
    (right, left) semiorthogonal decomposition of $\Cscr$} if
    \begin{enumerate}
        \item[(1)] the filtration is exhaustive and
        \item[(2)] for each arrow $p\in P$, the subcategory
            $\F_p\Cscr\subseteq\Cscr$ is (right, left) admissible.
    \end{enumerate}
\end{definition}

\begin{remark}
    In practice, while admissible $P$-shaped filtrations come up in many
    situations,
    $P$-shaped semiorthogonal decompositions occur usually when $P$ is
    filtered. Moreover, in this case,
    each $\F_p\Cscr\subseteq\Cscr$ is automatically admissible (resp. right-admissible,
    left-admissible) as well by~\cite[Proposition~4.4]{bondal-kapranov}, i.e.,
    condition (2) above is superfluous.
\end{remark}

We make a comparison between Definition~\ref{def:adm-fil} with a notion that
appears in textbook references (e.g. \cite[Definition 1.59]{huybrechts-fm}), at
least in the case of the finite ordered set $[n]$. See
also~\cite[Proposition~4.4]{bondal-kapranov}.

\begin{proposition} \label{prop:sod-adm}
     Let $\Cscr$ be an $\Oscr_S$-linear category. Then the following data are equivalent:
     \begin{enumerate}
         \item[{\rm (1)}] an $[n]$-shaped semiorthogonal decomposition
         $\F_\star\Cscr$ of $\Cscr$.
     \item[{\rm (2)}] A collection of admissible small $\Oscr_S$-linear full subcategories $\{ \Cscr_i \}_{0 \leq i \leq n}$ such that
    \begin{enumerate}
        \item[{\rm (a)}] for any $i \leq j$, $\MapSp_{\Cscr}(c_j,c_i) \we 0$ or
    all $c_i\in\Cscr_i$, $c_j\in\Cscr_j$, i.e., $\Cscr_i \subset (\Cscr_j)
            ^{\perp}$ and
\item[{\rm (b)}] the smallest stable subcategory containing the $\Cscr_i$'s is all of $\Cscr$.
    \end{enumerate}
    \end{enumerate}
\end{proposition}

\begin{proof} It suffices to consider the case of $n=1$. In this case, given a
    $[1]$-shaped semiorthogonal decomposition $\F_0\Cscr \hookrightarrow \F_1\Cscr \simeq \Cscr$ we can take the Verdier quotient $\frac{\F_{1}\Cscr}{\F_0\Cscr}$ which is canonically equivalent to the right orthogonal of $\F_0\Cscr$ in $\F_1\Cscr$ by Proposition~\ref{prop:all-equal} ((3)$\Rightarrow$(5) direction) and is an admissible subcategory of $\Cscr$ by assumption. So the collection of admissible subcategories $\{\F_0\Cscr, \frac{\F_{1}\Cscr}{\F_0\Cscr}\}$ satisfy 2(a). Since the filtration is assumed to be exhaustive, we get 2(b). 

Now assume that we have a collection $\{ \Cscr_0, \Cscr_1 \}$ as in (2). We define the filtration $\Cscr(\star):= \Cscr_0 \rightarrow \langle \Cscr_0, \Cscr_1\rangle$ where the $ \langle \Cscr_0, \Cscr_1\rangle$ indicate the smallest stable $\infty$-category containing both $\Cscr_0$ and $\Cscr_1$. Then the inclusion $\Cscr_0 \rightarrow \langle \Cscr_0, \Cscr_1\rangle$ has $\Cscr_1$ as the Verdier quotient and Proposition~\ref{prop:all-equal} ((5)$\Rightarrow$(3) direction) tells us the inclusion does have a right adjoint.
The fact that this filtration is exhaustive follows from 2(b), i.e., the equivalence $\langle \Cscr_0, \Cscr_1\rangle \simeq \Cscr$.
\end{proof}

\begin{remark} \label{rem:k} If $\Cscr$ is a stable $\infty$-category, we denote by
    $\K(\Cscr)$ the algebraic $K$-theory spectrum of $\Cscr$. An $[n]$-shaped
    semiorthogonal decomposition $\Cscr(\star) \rightarrow \Cscr$ induces a decomposition
    $\K(\Cscr) \simeq \prod_{p=0}^n \K(\frac{\F_{p}\Cscr}{\F_{p-1}\Cscr})$. However, for a
    general poset $P$ such a decomposition on $K$-theory is \emph{not} guaranteed.
\end{remark}

\subsection{Marked variants}

We also want to discuss marked variants of semiorthogonal decompositions. Of greatest
interest are markings by exceptional objects.

\begin{definition}
    Suppose that $\Cscr$ is a small $\Oscr_S$-linear stable $\infty$-category.
    An {\bf exceptional object} of $\Cscr$ is an object $e\in\Cscr$ such that
    $\MapSp_\Cscr(e,e)\we \Oscr_S$ as an $\Oscr_S$-algebra.\footnote{Here we abuse notation a little and view
    $\ShMap_\Cscr(x,x)$ as a perfect complex on $S$.} In this case, the thick
    $\Oscr_S$-linear subcategory of $\Cscr$ generated by $e$ is equivalent to $\Perfscr(S)$.
    
    Let $P$ be a poset. A collection of objects $\{ e_p \}_{p \in P}$ of
    $\Cscr$ is an {\bf exceptional sequence} if each $e_p$ is
    an exceptional object and if $\MapSp_{\Cscr}(e_q,e_p)\we 0$ for $p<q$. The exceptional sequence is
    called {\bf full} if the objects generate $\Cscr$ as an $\Oscr_S$-linear category. If $R$ is an ordinary commutative ring, then
    an exceptional sequence in an $R$-linear category is called {\bf strong} if
    $\Hom_{\Cscr}(e_p,e_q[n])=0$ for $n\neq 0$ and all $p,q\in P$.
\end{definition}

\begin{definition}\label{def:adm-fil-marked} Let $\Cscr$ be an $\Oscr_S$-linear stable
    $\infty$-category. Then a marked $P$-shaped (right, left) admissible
    filtration $(\F_\star\Cscr
    \rightarrow \Cscr, M_*)$ is {\bf exceptional} if
    for each $p \in P$ the object $M_p$ is exceptional. If $P=[n]$, the
    admissible filtration is an $[n]$-shaped semiorthogonal decomposition, and each
    $M_p$ is contained in
    $(\F_{p-1}\Cscr)_{\F_{p}\Cscr}^\perp$, we say that
    $(\F_\star\Cscr\rightarrow\Cscr,M_*)$ is an {\bf exceptionally marked $[n]$-shaped semiorthogonal
    decomposition}.
\end{definition}

Following~\cite{auel-bernardara}, we could also study $\Ascr$-exceptional objects for any Azumaya $\Oscr_S$-algebra
$\Ascr$. For these, we require $\ShMap_\Cscr(e,e)\we\Ascr$ so that $e$
generates a subcategory of $\Cscr$ equivalent to $\Perfscr(\Ascr)$. We could
then define a version of $\{\Ascr_p\}_{p\in P}$-exceptional sequences. Call
these {\bf twisted exceptional objects}. We will use the obvious notion of a
{\bf full twisted exceptional collection}, which gives rise to the notion of a twisted
exceptionally marked semiorthogonal decomposition, whose formulation we leave to the reader.

\subsection{The local nature of admissibility}

We now prove the following result, which is the main technical theorem of the paper.
If $S$ is an algebraic stack, $\Cscr$ is an $\Oscr_S$-linear category, and $T\rightarrow S$
is a morphism of algebraic stacks, we let $\Cscr_T=\Perfscr(T)\otimes_{\Perfscr(S)}\Cscr$.

\begin{theorem}\label{thm:et-sod}
    Let $S$ be a $1$-affine algebraic stack.
    Suppose that $\Cscr$ is an $\Oscr_S$-linear $\infty$-category and let
    $\Ascr\subseteq\Cscr$ be an $\Oscr_S$-linear inclusion of a full subcategory.
    If $U\rightarrow S$ is an fppf cover, then $\Ascr\rightarrow\Cscr$ is right-admissible if
    and only if $\Ascr_U\rightarrow\Cscr_U$ is right-admissible.
    The same holds for left-admissibility and admissibility.
\end{theorem}

The next proposition highlights the importance of working in an enhanced setting, as opposed to just with
triangulated categories. It states that the notion of right-admissibility is stable under (derived) base change.
This generalizes~\cite[Theorem~5.6]{kuznetsov}.

\begin{proposition}\label{prop:bcadmissible}
    Suppose that $i\colon\Ascr\subseteq\Cscr$ are $\Oscr_S$-linear categories, and assume
    that $i$ is right-admissible. If $\Fscr$ is another $\Oscr_S$-linear stable
    $\infty$-category, then
    \begin{equation*}
        \Fscr\otimes
        i\colon\Fscr\otimes_{\Perfscr(S)}\Ascr\rightarrow\Fscr\otimes_{\Perfscr(S)}\Cscr
    \end{equation*}
    is fully faithful and right-admissible. A similar statement holds for left-admissibility.
\end{proposition}

\begin{proof}
    Let $r\colon\Cscr\rightarrow\Ascr$ be the right adjoint to $i$. To say that $r$ and $i$
    are adjoint is the same as giving a natural transformation $\id_\Ascr\rightarrow r\circ i$ such that
    for each $a\in\Ascr$ and $b\in\Cscr$ the induced composition
    $$\eta_{a,b}\colon\ShMap_\Cscr(i(a),b)\rightarrow\ShMap_\Ascr(ri(a),r(b))\rightarrow\ShMap_\Ascr(a,r(b))$$
    is an equivalence.
    Moreover, if $i$ and $r$ are adjoint, $i$ is fully faithful if and only if the unit
    natural transformation is an equivalence.

    Let
    $i_\Fscr\colon\Fscr\otimes_{\Perfscr(S)}\Ascr\rightarrow\Fscr\otimes_{\Perfscr(S)}\Cscr$
    and
    $r_\Fscr\colon\Ascr\otimes_{\Perfscr(S)}\Cscr\rightarrow\Fscr\otimes_{\Perfscr(S)}\Ascr$
    be the maps induced by $i$ and $r$ by functoriality of the tensor product.
    We have an induced natural isomorphism $\id_{\Fscr\otimes_{\Perfscr(S)}\Ascr}\rightarrow
    r_\Fscr\circ i_\Fscr$ (since $\id_\Ascr\we r\circ i$). Thus, to prove the lemma, it is
    enough to prove that $i_\Fscr$ and $r_\Fscr$ are adjoint, since then fully faithfulness
    then follows from the discussion above.

    Consider for $a\in\Fscr\otimes_{\Perfscr(S)}\Ascr$ and
    $b\in\Fscr\otimes_{\Perfscr(S)}\Cscr$ the induced map
    \begin{equation}\label{eq:adj}
        \ShMap_{\Fscr\otimes\Cscr}(i_\Fscr(a),b)\rightarrow\ShMap_{\Fscr\otimes\Ascr}(r_\Fscr i_\Fscr(a),r_\Fscr(b))\rightarrow\ShMap_{\Fscr\otimes\Ascr}(a,r_\Fscr(b)).
    \end{equation}
    We would like to show that this is an equivalence. Suppose that $a=x\otimes y$ and
    $b=w\otimes z$ are pure tensors, i.e., $x,w\in\Fscr$, $y\in\Ascr$, and $z\in\Cscr$.
    Then, we find the map is equivalent to $\ShMap_\Fscr(x,w)\otimes\eta_{y,z}$ since, for
    example, $$\ShMap_{\Fscr\otimes\Cscr}(i_\Fscr(x\otimes y),w\otimes
            z)\we\ShMap_{\Fscr\otimes\Cscr}(x\otimes i(y),w\otimes
                z)\we\ShMap_\Fscr(x,w)\otimes_{\Oscr_Y}\ShMap_\Cscr(i(y),z).$$ Since
    $\eta_{y,z}$ is an equivalence, we see that~\eqref{eq:adj} is an equivalence for pure
    tensors. Since the pure tensors generate $\Fscr\otimes_{\Perfscr(S)}\Ascr$ and
    $\Fscr\otimes_{\Perfscr(S)}\Cscr$, a standard thick subcategory argument proves that~\eqref{eq:adj} is an
    equivalence for all $a\in\Fscr\otimes\Ascr$ and $b\in\Fscr\otimes\Cscr$.
\end{proof}

\begin{proof} [Proof of Theorem~\ref{thm:et-sod}] 
        The necessity of fppf local admissibility follows immediately from the stability of admissibility under base change by 
        Proposition~\ref{prop:bcadmissible} with $\Fscr=\Perfscr(U)$.
        It suffices to prove the converse for right-admissibility, since
        $\Ascr \subseteq\Cscr$ is left-admissible if and only if
        $\Ascr^{\op}\subseteq\Cscr^{\op}$ is right-admissible.

        Let $p\colon U \rightarrow S$ be an fppf cover and assume that
        $\Ascr_{U}\rightarrow\Cscr_{U}$ admits a right adjoint $R$.
        Let $\check{C}_\bullet(p)$ be the simplicial algebraic stack obtained by taking the \v{C}ech
        complex of $p$; i.e., $\check{C}_n(p)\iso U^{\times_S(n+1)}$. Taking
        $\Perfscr(\check{C}_\bullet(p))$ we obtain a cosimplicial $\Oscr_S$-linear category with
        $\Perfscr(\check{C}_n(p))\we\Perfscr(U^{\times_S(n+1)})$. Tensoring with the
        inclusion $i\colon\Ascr\rightarrow\Cscr$, we obtain a natural transformation
        $$\Ascr_{\check{C}_\bullet(p)}\rightarrow\Cscr_{\check{C}_\bullet(p)}$$ of
        cosimplicial $\Oscr_S$-linear categories.

        For simplicity, write $U_n=U^{\times_S(n+1)}$, the $n$th term in
        $\check{C}_\bullet(f)$. By Proposition~\ref{prop:bcadmissible}, each
        $i^n\colon\Ascr_{U_n}\rightarrow\Cscr_{U_n}$ admits a right adjoint, say $r^n$.

        \paragraph{Claim.}
        For each $q\colon[m]\rightarrow [n]$ in $\Delta$, the commutative square
        $$\xymatrix{
            \Ascr_{U_n}\ar[r]^{i^n}\ar[d]^{q_\Ascr^*}&\Cscr_{U_n}\ar[d]^{q_\Cscr^*}\\
            \Ascr_{U_m}\ar[r]^{i^m}&\Cscr_{U_m}
        }$$
        is right adjointable, i.e., the induced natural transformation $$q_\Ascr^*\circ
        r^n\rightarrow r^m\circ q_\Cscr^*$$ is an equivalence.\footnote{This natural
        transformation is called the Beck--Chevalley transformation and is
        constructed for example in~\cite[Definition~4.7.4.13]{ha}.} Indeed, this follows
        from Lemma~\ref{lem:etale-adj} below.

       To prove that $i$ admits a right adjoint, it suffices to check the object-wise criterion to be an adjoint
        by~\cite[5.2.7.8]{htt}.
        Namely, it is enough to show that for each $x\in\Cscr$ there exists an
        element $y\in\Ascr$ and a map $y\rightarrow x$ such that for each
        $w\in\Ascr$ the natural map
        $\ShMap_{\Ascr}(w,y)\rightarrow\ShMap_\Cscr(i(w),x)$ is an equivalence.

        By unstraightening, we can view the functor
        $\Ascr_{\check{C}_\bullet(p)}\colon\Delta\rightarrow\ShCat(S)$ as classifying a
        Cartesian fibration
        $$\widetilde{\Ascr}=\int_\Delta\Ascr_{\check{C}_\bullet(p)}\rightarrow\Delta^{\op}.$$
        Given $x\in\Cscr$ as above, the Beck--Chevalley transformations give $r^\bullet x$
        the structure of a section $r(x)$ of $\widetilde{\Ascr}$.\footnote{One way to make this precise
        is to use the classifying $2$-category of adjunctions, $\mathrm{Adj}$. Specifically,
        viewing $\ShCat(S)$ as a $2$-category where the mapping categories are the
            $\infty$-categories
            $\Fun^{\ex}(-,-)$ of exact $\Oscr_S$-linear functors, we have a forgetful
            functor $\Fun(\mathrm{Adj},\ShCat(S))\rightarrow\Fun(\Delta^1,\ShCat(S))$, where
            on the left $\Fun(\mathrm{Adj},\ShCat(S))$ is the $\infty$-category of
            $2$-categorical functors from $\mathrm{Adj}$ to $\ShCat(S)$. We note that the morphisms in $\Fun(\mathrm{Adj},\ShCat(S))$ are exactly given by squares which are pointwise adjointable \cite[Appendix A]{bispans}.
            Now, the theorem of
            Riehl--Verity~\cite{riehl-verity} implies that this functor is fully faithful
            with essential image exactly those objects of $\Fun(\Delta^1,\ShCat(S))$
                    possessing an adjoint. Now, we can view $i$ as defining a functor
                    $\Delta\rightarrow\Fun(\Delta^1,\ShCat(S))$. The pointwise
                        adjointability, implies that this functor factors through the subcategory to
                        give a functor $\Delta\rightarrow\Fun(\mathrm{Adj},\ShCat(S))$. If
        $\widetilde{\Cscr}\rightarrow\Delta^\op$ denotes the unstraightening of $\Cscr_{\check{C}_\bullet(p)}$,
        then this functor gives a functor $\widetilde{\Cscr}\rightarrow\widetilde{\Ascr}$ of
        $\infty$-categories over $\Delta^\op$. The object $x$ defines a (Cartesian) section
        of $\widetilde{\Cscr}\rightarrow\Delta^\op$ and we apply the functor to get a
        section of $\widetilde{\Ascr}\rightarrow\Delta$.} In other words, $r(x)$ is a kind of lax
        cosimplicial object. But, the claim above, that the squares are right adjointable,
        implies that this section is in fact Cartesian. Thus, $r(x)$ defines an object of
        $\Ascr$, which is equivalent to the $\infty$-category of Cartesian sections of
        $\widetilde{\Ascr}\rightarrow\Delta^\op$ by~\cite[Corollary~3.3.3.2]{htt}.
      
    We also have a map $ir(x)\rightarrow x$. Fix $w\in\Ascr$.
    We have
        \begin{eqnarray*}
        \ShMap_{\Ascr}(w,r(x)) & \we &
        \lim_{\Delta}\ShMap_{\Ascr_{U_\bullet}}(w_{U_\bullet},r^{\bullet}(x_{U_\bullet}))\\
         & \we &
         \lim_\Delta\ShMap_{\Cscr_{U_\bullet}}(\iota^{\bullet}(w_{U_\bullet}),x_{U_\bullet})\\
         & \we & \ShMap_\Cscr(i(w),x),
        \end{eqnarray*}
        as desired. 
\end{proof}

\begin{lemma} \label{lem:etale-adj} Suppose that $f: U \rightarrow S$ be a
    morphism of schemes. Let $F: \Dscr \rightarrow \Cscr$ be a functor of
    $\Oscr_S$-linear categories and consider the diagram
    \[
        \begin{tikzcd}
        \Dscr \ar{r}{F} \ar[swap]{d}{f_{\Dscr}^*} & \Cscr \ar{d}{f_{\Cscr}^*} \\
        \Dscr_{U} \ar{r}{F_U} & \Cscr_{U}.
        \end{tikzcd}
        \]
    If $F$ admits a right adjoint $G$, then the square above is right adjointable.             
\end{lemma}

\begin{proof}
    Suppose that $\{ y_i \}$ is a collection of objects in $\Dscr$ which
    generates $\Dscr$. Then the collection $\{ y_i \otimes \Oscr_U \}$ generates $\Dscr_U$. Hence to
    prove the claim, it suffices to prove that the canonical map
    \[
    \ShMap_{\Dscr_{U}}(y_i \otimes \Oscr_U, f^*_{\Dscr}Gx) \rightarrow
    \ShMap_{\Dscr_{U}}(y_i \otimes \Oscr_U, G_{U}f^*_{\Cscr}x),
    \]
    is an equivalence. This follows from the following computation
    \begin{eqnarray*}
    \ShMap_{\Dscr_{U}}(y_i \otimes \Oscr_U, f^*_{\Dscr}Gx) & = &
        \ShMap_{\Dscr_{U}}(y_i \otimes \Oscr_U, Gx \otimes \Oscr_U) \\
    & \simeq &  \ShMap_{\Dscr}(y_i, Gx ) \otimes \Oscr_U\\
    & \simeq &  \ShMap_{\Cscr}(Fy_i, x )  \otimes \Oscr_U\\
    & \simeq & \ShMap_{\Cscr_U}(F_Uy_i, x \otimes \Oscr_U )\\
    & \simeq & \ShMap_{\Dscr_{U}}(y_i \otimes \Oscr_U, G_{U}f^*_{\Cscr}x).
    \end{eqnarray*}
    Here, the only nontrivial step is the equivalence in the second line which follows from the
    same argument as in \cite[Lemma 2.7]{ag}.
\end{proof}

\begin{definition}\label{def:sodstack}
    Let $P$ be a poset. We denote by $\mathbf{Sod}_P$ the subprestack of
    $\mathbf{Filt}_P$ spanned by the $P$-shaped
    semiorthogonal decompositions. We call the prestack $\mathbf{Sod}_P$ the {\bf stack of
    $P$-shaped semiorthogonal decompositions}. 
\end{definition}

\begin{remark}
    Proposition~\ref{prop:bcadmissible} implies that $\mathbf{Sod}_P$ is indeed
    a prestack.
\end{remark}

Now, we see that $\mathbf{Sod}_P$ is a stack, which proves Theorem~\ref{thm:sods}.

\begin{corollary} \label{cor:sod-desc} The prestack $\mathbf{Sod}_P$ is an fppf stack.
\end{corollary}

\begin{proof}
    For each $\Spec R\rightarrow S$, we have an inclusion of connected components 
    $\mathbf{Sod}_P(R)\subseteq\Filt_P(R)$. Since $\mathbf{Filt}_P$ is an fppf stack by
    Theorem~\ref{thm:desc}, it suffices to check the effectiveness of descent, which is precisely
    Theorem~\ref{thm:et-sod}.
\end{proof}

\begin{corollary} \label{cor:main-sod-version} 
    Let $P$ be a poset, $S$ a
    qcqs scheme. Let $G$ be a flat affine algebraic $S$-group scheme of finite
    presentation
    and let $X$ be a qcqs $S$-scheme with an action of $G$.
    Suppose that $\Perfscr(X)$ admits a semiorthogonal decomposition
    $\F_\star\Cscr\rightarrow\Cscr$. If $G$ preserves the filtration
    $\F_\star\Cscr\rightarrow\Cscr$, then the induced filtration
    $\F_\star\Perfscr([X/G])\rightarrow\Perfscr([X/G])$ is a $P$-shaped
    semiorthogonal decomposition.
\end{corollary}

\begin{proof}
    This follows immediately from Theorem~\ref{thm:et-sod} since
    $X\rightarrow[X/G]$ is an fppf cover.
\end{proof}

\begin{remark}
    Under the equivalence given by Proposition~\ref{prop:sod-adm} between our form of semiorthogonal decompositions and
    the usual form, Corollary~\ref{cor:main-sod-version} gives a very general
    form of Elagin's theorem.
\end{remark}


\section{The twisted Brauer space for filtrations} \label{sec:tbs}

We introduce in this section tools for constructing obstruction classes in cohomology
attached to filtrations.

\subsection{Twisted Brauer spaces: recollections}

Throughout, $S$ will be a $1$-affine algebraic stack and $\Cscr$ will be an $\Oscr_S$-linear
category. The twisted Brauer space of $\Cscr$ is a topological
space (or, really, simplicial set) $\ShBr^\Cscr(S)$ whose points classify
idempotent complete $S$-linear stable $\infty$-categories $\Dscr$ which are
fppf locally on $S$ equivalent to $\Cscr$, i.e., a {\bf twisted form} of
$\Cscr$. Here are some salient features of twisted Brauer spaces; we refer to
\cite{antieau-tbs} for proofs. Note that~\cite{antieau-tbs} treats the \'etale
case; thanks to Theorem~\ref{thm:lurie}, the same results work for the fppf
analogue.

\begin{enumerate} 
\item[\rm (i)] If $\Cscr = \Perfscr(S)$, then $\ShBr^{\Cscr} \simeq \ShBr$ is
    the fppf version of the
Brauer space considered in~\cite{ag,ToenderAz}; in particular
\[
\pi_i\ShBr(S) \iso\begin{cases}
 \H^1_{\fppf}(S;\ZZ) \times \H^2_{\fppf}(S;\GG_m)&\text{if $i=0$,}\\
 \H^0_{\fppf}(S;\ZZ) \times \H^1_{\fppf}(S;\GG_m)&\text{if $i=1$,}\\
 \H^0_{\fppf}(S;\GG_m)&\text{if $i=2$, and}\\
 0&\text{otherwise,}
 \end{cases}
\]
where the higher homotopy groups are computed at the basepoint
        $\Perfscr(S)\in\ShBr(S)$.\footnote{Recall Grothendieck's
        theorem~\cite[Section 5, Th\'eor\`eme~11.7]{grothendieck-brauer-3} that
        for a smooth affine group scheme $G$ the natural map
        $\H^s_\et(S,G)\rightarrow\H^s_\fppf(S,G)$ is an isomorphism for all
        $s$.}
\item[\rm (ii)] There is a one-to-one correspondence between $\Oscr_S$-linear
    equivalence classes of
$\Oscr_S$-linear categories which are fppf locally equivalent to $\Cscr$ and the elements
of the set $\pi_0\ShBr^\Cscr(S)$. Given a point $\Dscr$ of
$\ShBr^\Cscr(S)$, the higher homotopy groups of the space $\ShBr^\Cscr(S)$ are
given by $$\pi_i(\ShBr^\Cscr(S),\Dscr)\iso\pi_{i-1}\ShAut_{\Dscr}(S),$$ where
        $\ShAut_\Dscr(S)$ is the space of derived $\Oscr_S$-linear autoequivalences of $\Dscr$. 
\item[\rm (iii)] The space $\ShAut_\Dscr(S)$ is the space of global sections of a sheaf of
spaces $\ShAut_\Dscr$. If $S=\Spec R$ is affine, then the higher homotopy groups of this
space are well-understood:
$$\pi_{i-1}\ShAut_\Dscr(S)\iso\begin{cases}
    \HH^0(\Dscr/S)^\times&\text{$i=2$,}\\
    \HH^{2-i}(\Dscr/S)&\text{$i\geq 3$,}
\end{cases}$$
where $\HH(\Dscr/S)$ denotes the Hochschild cohomology of $\Dscr$ relative to
$S=\Spec R$. This is a result of  To\"en's~\cite[Corollary~1.6]{toen-dg}. For
non-affine $S$, one computes the space $\ShAut_\Dscr(S)$ via a local-global spectral sequence. 
Note that if $\Dscr\we\Perfscr(X)$ for some scheme $X$, then
$\HH^{2-i}(\Dscr/S)=0$ for $i\geq 3$: schemes do not have non-zero
negative Hochschild cohomology groups.
\item[\rm (iv)] the twisted Brauer space is the space of $S$-sections of an
    fppf sheaf 
\[
    \ShBr^\Cscr\colon S_{\fppf} \rightarrow \Sscr.
\]
Since every point of $\ShBr^\Cscr$ is fppf-locally equivalent to
$\Cscr$, the sheaf $\ShBr^\Cscr$ is connected as a sheaf of
spaces. Thus, the twisted Brauer space is the classifying stack of the
sheaf of groups $\ShAut_{\Cscr}$, formalizing the way in which it classifies
twisted forms of $\Cscr$. Here, the sheaf of groups should be understood in the
        homotopical context: $\ShAut_\Cscr$ is like a sheaf of $H$-spaces and
        is precisely a sheaf of what are called grouplike $A_\infty$-spaces.
        Using this perspective, one can often enumerate twisted
forms using descent spectral sequences; see \cite{antieau-tbs} for some
examples.
\end{enumerate}

\subsection{Twisted Brauer spaces for filtrations}

We now consider a variant of twisted Brauer spaces where $\Cscr$ is equipped
with a filtration.

\begin{definition}
    Let $P$ be a poset and let $\F_\star\Cscr\rightarrow\Cscr$ be a $P$-shaped
    filtration on $\Cscr$. We let $\ShBr^{\F_\star\Cscr}$ be the fppf
    sheafification of the corresponding point of $\iota\ShFilt_P$. This is the
    {\bf twisted Brauer space of the filtration}
    $\F_\star\Cscr\rightarrow\Cscr$.

    Similarly, if $(\F_\star\Cscr\rightarrow\Cscr,M_*)$ is a marked $P$-shaped
    filtration, we let the {\bf twisted Brauer space of the marked filtration}
    $\ShBr^{(\F_\star\Cscr,M_*)}$ be the fppf sheafification of the corresponding
    point in $\iota\mathbf{MFilt}_P$.
\end{definition}

\begin{remark}\label{rem:zero}
    We remark that setting $M_i = 0$ for all $i \in P$ is equivalent to having
    no markings on the filtration so that $\ShBr^{\F_\star\Cscr}\simeq
    \ShBr^{(\F_\star\Cscr,M_*)}$ in this case.
\end{remark}

By construction both $\ShBr^{\F_\star\Cscr}$ and
$\ShBr^{(\F_\star\Cscr,M_*)}$ are fppf sheaves on the big
site which, \emph{a priori}, take values in the $\infty$-category of large
spaces.
We will see in fact that they are sheaves of small spaces.

We record some basic properties of these gadgets bearing in mind Remark~\ref{rem:zero}.

\begin{proposition}\label{prop:main}
    Let $S$ be a $1$-affine algebraic stack, $\F_\star\Cscr\rightarrow\Cscr$ a
    $P$-shaped filtration, and $M_*$ a marking for the filtration.
    The fppf sheaves $\ShBr^{(\F_\star\Cscr,M_\ast)}$ satisfy the
    following properties.
    \begin{enumerate}
        \item[{\rm (a)}] For any qcqs $S$-scheme
            $T$, the space $\ShBr^{(\F_\star\Cscr,M_\ast)}(T)$ is the space of
            marked filtered idempotent complete $T$-linear stable
            $\infty$-categories $(\F_\star\Cscr,N_\ast)$ which are
            fppf-locally equivalent to $(\F_\star\Cscr\otimes_{\Perfscr(S)}\Perfscr(T),M_\ast\otimes_{\Oscr_S}\Oscr_T)$.
        \item[{\rm (b)}] Let $\ShAut_{(\F_\star\Cscr,M_\ast)}$
            denote the sheaf of automorphisms of the marked $P$-filtered stable
            $\infty$-category defined by $\F_\star\Cscr\rightarrow\Cscr$ and
            $M_*$. There is a natural equivalence
            $$\ShB\ShAut_{(\F_\star\Cscr,M_\ast)}\we\ShBr^{(\F_\star,M_\ast)}$$
            of fppf sheaves.
        \item[{\rm (c)}] There are forgetful maps  \[
           \ShBr^{(\F_\star\Cscr,M_\ast)}\rightarrow
            \ShBr^{\F_\star\Cscr} \rightarrow \ShBr^{\Cscr}
           \] of fppf sheaves. A
            twisted form $\Dscr$ of $\Cscr$ giving a point of
            $\ShBr^\Cscr(T)$ lifts to $ \ShBr^{\F_\star\Cscr}(T)$ (resp.
            $\ShBr^{(\\F_\star\Cscr,M_\ast)}(T)$) if
            and only if the (marked) filtration on $\Cscr$ descends to $\Dscr$.

            \item [{\rm (d)}] The morphism of sheaves
            $\ShBr^{\F_\star\Cscr}\rightarrow\ShBr^{\Cscr}$ is $0$-truncated, i.e, the
            fibers are fppf sheaves of sets.
            
            \item[{\rm (e)}] On sheaves of fundamental groups, the inclusion
            \[
            \pi_1(\ShBr^{\F_\star\Cscr}, \F_\star\Cscr) \hookrightarrow \pi_1(\ShBr^{\Cscr},\Cscr) \simeq \pi_0(\ShAut_{\Cscr}),
            \]
            corresponds to the inclusion of those automorphisms of $\Cscr$
            which preserve the filtration $\F_\star\Cscr$.
                    
    \end{enumerate}
\end{proposition}

\begin{proof} 
    \begin{enumerate}
        \item[{\rm (a)}] Let $\Gscr \subset \mathbf{MFilt}_P$ be the
            subpresheaf classifying objects $(\F_\star\Dscr \rightarrow \Dscr,
            M_{\ast})$ which are fppf locally equivalent to $(\F_\star\Cscr
            \rightarrow \Cscr, M_\ast)$. By definition of
            $\ShBr^{(\F_\star\Cscr,M_*)}$ the map
            $\ShBr^{(\F_\star\Cscr,M_*)}
            \rightarrow\mathbf{MFilt}_P$ factors through the
            inclusion of $\Gscr$. Furthermore, the subpresheaf $\Gscr$ is
            actually an fppf sheaf since the condition of a section of $
            \mathbf{MFilt}_P$ being in $\Gscr$ is fppf local. Since, fppf
            locally, the map $\ShBr^{(\F_\star\Cscr,M_*)}
            \rightarrow \Gscr$ is an equivalence on sections we get an
            equivalence of fppf sheaves.
        \item[{\rm (b)}] We note that the fppf sheaf of connected components
            $\pi_0(\ShB\ShAut_{(\F_\star\Cscr,M_\ast)})$ is
            equivalent to the terminal sheaf since, fppf locally,
            $(\F_\star\Cscr,M_\ast)$ is the only equivalence
            class of objects. We thus have an inclusion of sheaves
            $\ShB\ShAut_{(\F_\star\Cscr,M_\ast)} \hookrightarrow
            \ShBr^{(\F_\star\Cscr,M_\ast)}$ where the former
            is the maximal subgroupoid of the full subcategory spanned by the
            global base point, $(\F_\star\Cscr,M_\ast)$. Hence
            to prove the claimed equivalence it suffices to apply $\Omega$ to
            both sides (which preserves sheaves). The resulting sheaf in both
            cases is the sheaf of automorphisms
            $\ShAut_{(\F_\star\Cscr,M_\ast)}$. 
        \item[{\rm (c)}] This follows by definition.
        
        \item[{\rm (d)}] This follows from Lemma~\ref{lem:0category}.   
        
        \item[{\rm (e)}] Let the fiber $\ShBr^{\F_\star\Cscr}\rightarrow\ShBr^{\Cscr}$ be denoted by $\Fscr$.  From part
            $(d)$, we have an exact sequence of sheaves:
            \[
            0 \simeq \pi_2(\Fscr) \rightarrow  \pi_1(\ShBr^{\F_\star\Cscr},
            \F_\star\Cscr) \hookrightarrow \pi_1(\ShBr^{\Cscr},\Cscr) \simeq \pi_0(\ShAut_{\Cscr}),
            \]
            where the last isomorphism is \cite[Lemma 2.5]{antieau-tbs}. The
            claim then follows from (b) and the definition of the forgetful
            maps.
    \end{enumerate}
\end{proof}

\begin{remark} Out of the above properties, property (d) is the most striking:
    it tells us that in order to descend a filtration on $\Cscr$ to a
    filtration on its twisted form the obstruction is purely ``discrete." This
    explains how previous results on descent for semiorthogonal decompositions,
    such as \cite{elagin} or \cite{bergh-schnurer}, could avoid the subtleties of
    gluing higher categorical objects (such as dg categories); see
    Theorem~\ref{thm:main-disc} for a precise statement.
\end{remark}

We also note the following simple corollary, which says that there are not too
many twisted forms.

\begin{corollary} \label{cor:small} $\ShBr^{(\F_\star\Cscr,M_\ast)}$ is a sheaf of small spaces.
\end{corollary}

\begin{proof}
    After Proposition~\ref{prop:main}(b), the argument is the same as in
    \cite[Proposition 2.6]{antieau-tbs}.
\end{proof}

\subsection{Descending filtrations on twisted forms}\label{sec:desctwist}

We will use the language of twisted Brauer spaces to explain the following phenomena: to check if filtrations on a scheme induce a compatible filtration on its twisted form one only needs to check 1-categorical compatibilities. This leads to computability of the obstructions in trying to descend semiorthogonal decompositions as illustrated in our examples in Section~\ref{sec:ex}.

\begin{theorem} \label{thm:main-disc}
    Let $S$ be a $1$-affine algebraic stack.
    Suppose that $X$ and $Y$ are two qcqs $S$-schemes. Let 
        \[\F_{\star}\Perfscr(X) \rightarrow
        \Perfscr(X)
        \] be a $P$-shaped filtration for some poset $P$. Assume further that 
    \begin{enumerate}
        \item[{\rm (i)}] there is a surjective fppf morphism $T \rightarrow S$ and an
            isomorphism $\alpha\colon Y_T \rightarrow X_T$ and
        \item[{\rm (ii)}] the induced filtration $\F_\star\Perfscr(Y_T)=\alpha^*q^*\F_\star\Perfscr(X)$ on
            $\Perfscr(Y_T)$ satisfies the cocycle condition:
            if $p_i\colon Y_{T\times_ST}\rightarrow Y_T$ are the two
            projections, then for each $p\in P$ the subcategories
            $$p_1^*\F_p\Perfscr(Y_T)\quad\text{and}\quad
            p_2^*\F_p\Perfscr(Y_T)$$ of $\Perfscr(Y_{T\times_ST})$ coincide.
        \end{enumerate}
        Then, there is a filtration $\F_{\star}\Perfscr(Y) \rightarrow
        \Perfscr(Y)$ and an equivalence of filtrations
        \[
        (\F_{\star}\Perfscr(Y_T) \rightarrow \Perfscr(Y_T)) \simeq (\F_{\star}\Perfscr(X_T) \rightarrow \Perfscr(X_T)),
        \]
        induced by $\alpha$. Moreover, if $\F_\star\Perfscr(X)$ is a $P$-shaped
        semiorthogonal decomposition of $\Perfscr(X)$, then the induced
        filtration $\F_\star\Perfscr(Y)$ is a semiorthogonal decomposition of
        $\Perfscr(Y)$.
\end{theorem}

\begin{proof} By assumption, the $S$-scheme $Y$ gives a global section
    $\beta_Y\colon S \rightarrow \ShBr^{\Perf(X)}$. By
    Proposition~\ref{prop:main}(c), we need to lift $\beta_Y$ along the map
    $\ShBr^{F_{\star}\Perf(X)} \rightarrow \ShBr^{\Perf(X)}$. We denote by
    $\Fscr_{\beta_Y}$ the fiber of $\ShBr^{F_{\star}\Perf(X)} \rightarrow
    \ShBr^{\Perf(X)}$ over $\beta_Y$, which is an fppf sheaf on $\Sch_{Y}$
    which is naturally equivalent to $\mathbf{Filt}_P^{\Cscr
    \otimes_{\Perf(X)} \Perf(Y)}$. We proceed to construct a section of the
    canonical map $\Fscr_{\beta_Y} \rightarrow Y$. But, by
    Proposition~\ref{prop:main}(d) (or Lemma~\ref{lem:0category}) the sheaf
    $\Fscr_{\beta_Y}$ is an fppf
    sheaf of sets, hence the hypothesis in (ii) suffices to construct the
    desired section. The final claim follows from Theorem~\ref{thm:et-sod}.
\end{proof}


\section{Examples} \label{sec:ex}
In many good situations, we have a fiber sequence
\[
\ShBr^{\F_\star\Cscr} \rightarrow \ShBr^{\Cscr} \rightarrow \ShB G,
\]
of sheaves
where, by Proposition~\ref{prop:main}(d), $G$ is a discrete group. In fact,
this will hold if and only if the sheaf of subgroups of the sheaf $\pi_0\ShAut_\Cscr$
corresponding to automorphisms which preserve the filtration $\F_\star\Cscr$ is
normal. (For a case where this does not happen, see
Example~\ref{ex:nonnormal}.)
Furthermore, Proposition~\ref{prop:main}(e) in conjunction with known
computations of the homotopy automorphisms of $\Perfscr(S)$ will let us compute
$G$. In these cases, we get a theory of \emph{characteristic classes} for
filtrations  --- to a twisted form $\Dscr$ of $\Cscr$ on $S$ we get a class $o(\Cscr(\star))
\in \H^1_{\fppf}(S, \G)$ whose vanishing controls whether or not we obtain a
filtration on $\Dscr$ which is fppf-locally equivalent to $\F_\star\Cscr$. We will illustrate how
this works in this section.

\begin{example}\label{ex:nonnormal}
    For a case where the sheaf of subgroups
    $\pi_0\ShAut_{\F_\star\Cscr}\subseteq\pi_0\ShAut_\Cscr$ is not normal,
    let $\Cscr\we\bigoplus_{i=1}^3\Perfscr(S)$ and consider the $[1]$-shaped filtration given by
    $\F_0\Cscr=\oplus_{i=1}^2\Perfscr(S)$ and $\F_1\Cscr=\Cscr$. Then,
    there is an exact sequence
    $0\rightarrow\ZZ^{\oplus 3}\rightarrow\pi_0\ShAut_\Cscr\rightarrow
    S_3\rightarrow 0$, where $S_3$ is the symmetric group on $3$ letters. Here, the
    $\ZZ^{\oplus 3}$ corresponds to separately suspending in each copy of
    $\Perfscr(S)$. There is also an
    exact sequence $$0\rightarrow\ZZ^{\oplus
    3}\rightarrow\pi_0\ShAut_{\F_\star\Cscr}\rightarrow\ZZ/2\rightarrow 0.$$ Since
    $\ZZ/2$ is not normal in $S_3$, we see that $\pi_0\ShAut_{\F_\star\Cscr}$ is
    not normal in $\pi_0\ShAut_\Cscr$.
\end{example}

\subsection{From Beilinson to Bernardara}\label{sect:beilber}

Let $[n]$ be the poset $\{0\rightarrow
1\rightarrow\cdots\rightarrow i \rightarrow \cdots n\}$. Beilinson's description of the
derived category of $\PP_S^n$ gives an $[n]$-shaped
filtration of $\Perfscr(\PP_S^n)$ with
$$i \mapsto \F_i\Perfscr(\PP_S^n)=\langle\Oscr,\Oscr(1),\ldots,\Oscr(i)\rangle.$$ 
We call this the {\bf Beilinson filtration}.

\begin{lemma} \label{lemma:pn} There is a fiber sequence
    \[
    \ShBr^{\F_\star\Perfscr(\PP_S^n)}\rightarrow\ShBr^{\Perfscr(\PP_S^n)} \rightarrow \ShB\ZZ
    \]
    of fppf sheaves of spaces on $S$.
\end{lemma}

\begin{proof} Since, by a result of Bondal and Orlov  \cite[Theorem~3.1]{bondal-orlov}, we know the homotopy automorphisms of $\Perfscr(\PP^n_S)$, the sheaf of automorphisms of $\Perfscr(\PP^n_S)$ has homotopy sheaves
    $$\pi_i\ShAut_{\Perfscr(\PP^n_S)}\iso\begin{cases}
        \PGL_{n+1} \times \ZZ \times \Pic_{\PP^n/S} &\text{if $i=0$,}\\
    \Gm            &\text{if $i=1$, and}\\
    0&\text{otherwise.}\end{cases}$$
    Here, $\PGL_{n+1}$ acts by automorphisms of the $S$-scheme $\PP^{n}_S$,
    the group $\ZZ$ acts by suspension $\Fscr\mapsto\Fscr[1]$ in
    $\Perfscr(\PP^{n}_S)$, and the relative Picard scheme
    $\Pic_{\PP^n/S}\iso\ZZ$ acts
    by tensoring with line bundles. 
    Similarly, for the Beilinson filtration $\F_\star\Perfscr(\PP^n_S)$, the sheaf of automorphisms is the subsheaf of groups on the connected components that preserve the filtration. By Proposition~\ref{prop:main}.e, this eliminates only the non-zero elements
    of $\Pic_{\PP_S^n/S}$ since they do not preserve the Beilinson filtration.
    Thus, $\ShAut_{\F_\star\Perfscr(\PP_S^n)}$ has homotopy sheaves
    $$\pi_i\ShAut_{\F_\star\Perfscr(\PP^n_S)}\iso\begin{cases}
    \PGL_{n+1} \times \ZZ&\text{if $i=0$,}\\
    \Gm            &\text{if $i=1$, and}\\
    0&\text{otherwise.}\end{cases}$$
    It follows that there is a fiber sequence $$\ShAut_{\F_\star\Perfscr(\PP_S^n)}\rightarrow\ShAut_{\Perfscr{\PP_S^n}}\rightarrow\ZZ,$$ and this sequence deloops to give a fiber sequence
    as claimed.
\end{proof}
            
The next theorem generalizes the main result of \cite{bernardara}.

\begin{theorem} \label{thm:sb} Let $S$ be a qcqs $1$-affine algebraic stack and
    let $P \rightarrow S$ be a Severi--Brauer scheme
    associated to an Azumaya algebra $\Ascr$ of degree $(n+1)$ with
    Brauer class $\alpha$. There exists a natural semiorthogonal
    decomposition
    \[
    \Perfscr(P) \simeq \langle \Perfscr(S),\Perfscr(S,\alpha),\cdots, \Perfscr(S, \alpha^{\otimes n}) \rangle.
    \]
\end{theorem}

\begin{proof}
    Let $f\colon S \rightarrow \ShBr^{\Perfscr(\PP^n_S)}$ classify $\Perfscr(P)$. By
    Lemma~\ref{lemma:pn}, to descend the Beilinson filtration to $\Perfscr(P)$,
    it suffices to prove that the composite $S \rightarrow \ShB\ZZ$ is null.
    But $f$ factors through the map $\ShBPGL_{n+1} \rightarrow
    \ShBr^{\Perfscr(\PP^n_X)}$ which classifies $\Perfscr$ of the universal
    $\PGL_{n+1}$-torsor. It suffices to prove that \[
    \H^1_{\fppf}(\ShBPGL_{n+1}; \ZZ) = 0
    \]
    in the case where $S=\Spec\ZZ$.
    We may use the spectral sequence $$\E_1^{st}=\H^s_\fppf(\PGL_{n+1}^{\times
    t},\ZZ)\Rightarrow\H^{s+t}_{\fppf}(\ShBPGL_{n+1},\ZZ)$$ associated to the simplicial scheme
    \[
    \begin{tikzcd}
    *
    & \PGL_{n+1} \arrow[l, shift left]
    \arrow[l, shift right]
    & \PGL_{n+1} \times \PGL_{n+1} \arrow[l]
    \arrow[l, shift left=
    2
    ]
    \arrow[l, shift right=
    2
    ]
    & \cdots
    \arrow[l, shift left]
    \arrow[l, shift right]
    \arrow[l, shift left=3]
    \arrow[l, shift right=3],
    \end{tikzcd}
    \]
    to compute fppf cohomology. The only groups that might contribute to
    $\H^1_{\fppf}(\ShBPGL_{n+1},\ZZ)$  from the $\E_1$-page
    are $\E_1^{0,1}=\H^0_{\fppf}(\PGL_{n+1};\ZZ)$ and $\E_1^{1,0}=\H^1_{\fppf}(\ast;\ZZ)=0$. The latter is zero since
    $\ast=\Spec\ZZ$ is normal (see~\cite[2.1]{deninger}). On the other hand,
    $\H^0_\fppf(\PGL_{n+1},\ZZ)\iso\ZZ$. Now, in low degrees, the $d_1$-differentials give
    a complex $$\E_1^{0,0}\rightarrow\E_1^{0,1}\rightarrow\E_1^{0,2}$$ and the
    cohomology in the middle term is $\E_2^{0,1}$. One can check easily that this complex is isomorphic
    to $\ZZ\xrightarrow{0}\ZZ\xrightarrow{\iso}\ZZ$. Hence, $\E_2^{0,1}=0$ so
    that $\H^1_{\fppf}(\ShBPGL_{n+1},\ZZ)=0$. This proves that we can
    descend the Beilinson filtration to obtain a $\Delta^n$-shaped
    semiorthogonal decomposition $\F_\star\Perfscr(P)$ of $\Perfscr(P)$.

    Each graded piece $$\frac{\F_p\Perfscr(P)}{\F_{p-1}\Perfscr(P)}$$ is a twisted
    form of $\Perfscr(S)$ and thus of the form $\Perfscr(S,\beta_p)$ for some
    $\beta_p$. To complete the theorem, it suffices to see that $\beta_p=\alpha^{\otimes p}$.
    But, we see by reducing to the universal case $S=\Spec\ZZ$ that
    $\beta_p=\alpha^{\otimes a_p}$ for some $a_p$ and then we can find $a_p=p$,
    for example by referring to Bernardara~\cite{bernardara}.
\end{proof}

\subsection{Marking the Beilinson filtration}

Now consider the marked version of the above picture, where we mark
$\F_p\Perfscr(\PP^n)$ by $\Oscr(p)$. We can in fact describe the sheaf
$\ShBr^{(\F_\star\Perfscr(\PP^n),\Oscr(\ast))}$. To do so, consider the
maximal torus $T_{\PGL_{n+1}}$ of $\PGL_{n+1}$.

\begin{proposition} \label{prop:rigidifies}
    \begin{enumerate}
        \item[\rm (i)] We have an equivalence of sheaves
            $$\ShBr^{(\F_\star\Perfscr(\PP^n),\Oscr(\ast))}\we\ShB G,$$ where $G$ is a central extension of $\PGL_{n+1}$ by its maximal torus, i.e., we have an exact sequence of groups
            \[
            1 \rightarrow T_{\PGL_{n+1}} \rightarrow G \rightarrow \PGL_{n+1} \rightarrow 1
            \]
            and $T_{\PGL_{n+1}}$ is in the center of $G$.
        \item[{\rm (ii)}] If $(\F_\star\Dscr,M_*)$ is a twisted form of
            $(\F_\star\Perfscr(\PP^n),\Oscr(\ast))$, then $\F_\star\Dscr$ is
            equivalent to the Beilinson filtration on $\Perfscr(P)$ for some
            Severi--Brauer scheme $P\rightarrow S$ associated to an
            Azumaya algebra of degree $n+1$.
        \item[\rm (ii)] If $P\rightarrow S$ is a Severi--Brauer scheme
            associated to an Azumaya algebra $\Ascr$ of degree $n+1$ on $S$, then
            $\F_\star\Perfscr(P)\in\pi_0\ShBr^{\F_\star\Perfscr(\PP^n)}(S)$
            lifts to $\pi_0\ShBr^{(\F_\star\Perfscr(\PP^n),\Oscr(\ast))}$ if
            and only if $\Ascr\iso\Escr\mathrm{nd}(\Vscr)$ for some vector
            bundle $\Vscr$ on $S$.
    \end{enumerate}
\end{proposition}

\begin{proof}
    By construction, there is a fiber sequence
    $$\ShAut_{(\F_\star\Perfscr(\PP^n),\Oscr(\ast))}\rightarrow\ShAut_{\F_\star\Perfscr(\PP^n)}\rightarrow\prod_{0\leq
    p\leq n}\iota\F_p\Perfscr(\PP^n)$$
    of sheaves of spaces where the right map sends a filtered automorphism
    $\varphi$ to the $(n+1)$-tuple
    $(\varphi(\Oscr(0)),\ldots,\varphi(\Oscr(n)))$. The left-hand term is the
    fiber over $(\Oscr(0),\ldots,\Oscr(n))$. It follows that there is an exact
    sequence (of sheaves of abelian groups)
    \begin{gather*}
        0\rightarrow\pi_1\ShAut_{(\F_\star\Perfscr(\PP^n),\Oscr(*))}\rightarrow\pi_1\ShAut_{\F_\star\Perfscr(\PP^n)}\xrightarrow{a}\prod_{0\leq
    p\leq
        n}\Gm\\\rightarrow\pi_0\ShAut_{(\F_\star\Perfscr(\PP^n),\Oscr(\ast))}\xrightarrow{b}\pi_0\ShAut_{\F_\star\Perfscr(\PP^n)}.
    \end{gather*}
    We already know that $\pi_1\ShAut_{\F_\star\Perfscr(\PP^n)}\iso\Gm$,
    which appears as the
    natural automorphisms of the identity on $\Perfscr(\PP^n)$. These
    natural isomorphisms of $\Perfscr(\PP^n)$ act as $\Gm$
    on each $\Oscr(p)$. Thus, the map $a$ in the diagram is the diagonal
    embedding of $\Gm$ in $\prod_{0\leq p\leq n}\Gm$. It follows that
    $\pi_1\ShAut_{(\F_\star\Perfscr(\PP^n),\Oscr(\ast))}=0$. We also can see
    directly that the image of the map $b$ in
    $\pi_0\ShAut_{\F_\star\Perfscr(\PP^n)}\iso\PGL_{n+1}\times\ZZ$ is
    $\PGL_{n+1}$ since the copy of $\ZZ$ appears as the suspension operation on
    $\Perfscr(\PP^n)$ which does not preserve the marking.
    This proves part (i).
   
    For part (ii), we see from the map $G\rightarrow\PGL_{n+1}$ that any
    twisted form $(\F_\star\Dscr,N_*)$ of
    $(\F_\star\Perfscr(\PP^n),\Oscr(\ast))$ has the property that
    $\F_\star\Dscr$ is equivalent to the Beilinson filtration on a
    Severi--Brauer scheme $P\rightarrow S$ for a degree $(n+1)$ Azumaya algebra
    over $S$.

    Using the exact sequence $1\rightarrow T_{\PGL_{n+1}}\rightarrow
    G\rightarrow\PGL_{n+1}\rightarrow 1$, we see that a lift of a class
    $P\in\H^1_\fppf(S,\PGL_{n+1})$ to $\H^1_\fppf(S,G)$ exists if and only if
    the obstruction class $\mathrm{ob}(P)\in\H^2_\fppf(S,T_{\PGL_{n+1}})$
    vanishes. We leave it to the reader to check that under the natural
    isomorphism $\prod_{0< p\leq n}\Gm\iso T_{\PGL_{n+1}}$ the obstruction class
    of $P$ is $(\alpha,\alpha^{\otimes 2},\ldots,\alpha^{\otimes n})$.
    Thus, a lift exists if and only if the Azumaya algebra $\Ascr$ has trivial
    Brauer class, which happens if and only if
    $\Ascr\iso\Escr\mathrm{nd}(\Vscr)$ for some vector bundle $\Vscr$, which is
    what we wanted to prove.
\end{proof}
            
In particular, we recover the well-known fact that if the marked filtration descends to the
Severi--Brauer variety of an Azumaya algebra $\Ascr$, then $\Ascr$
is the sheaf of endomorphisms of a vector bundle.

\subsection{Involution surfaces}

Now we study twisted forms of $\PP^1\times\PP^1$.

\begin{definition}
    An {\bf involution surface} is an $S$-scheme $X$ which is fppf-locally isomorphic to $\PP^1_S\times_S\PP^1_S$.
\end{definition}

Involution surfaces are classified by pairs
$(T,\Ascr)$ consisting of
a $\ZZ/2$-Galois extension $T$ of $S$ and a quaternion Azumaya algebra
$\Ascr$ over $T$. The associated surface $X(T,\Ascr)$ is
$\mathrm{Re}_{T/S}\SB(\Ascr)$, the Weil
restriction from $T$ to $S$ of the Severi--Brauer variety of $\Ascr$.
See~\cite[Example~3.3]{auel-bernardara} or~\cite[Section~15.B]{involutions}.

\begin{example}\label{ex:pairs}
    If the quadratic extension $T$ is split, then $T=S\coprod S$ and
    $\Ascr=\Ascr_1\times \Ascr_2$, where $\Ascr_1$ and $\Ascr_2$ are quaternion
    Azumaya
    algebras over $S$. In this case, $X(T,\Ascr)\iso\SB(\Ascr_1)\times\SB(\Ascr_2)$, the
    product of the Severi--Brauer schemes of $\Ascr_1$ and $\Ascr_2$.
\end{example}

Since $\PP^1\times\PP^1$ is the zero-locus of a quadric form in four variables,
its
automorphism group is the smooth algebraic group $\mathrm{PO}_4$.
Now, note that there is a natural inclusion
$$\PGL_2\times\PGL_2\hookrightarrow\mathrm{PO}_4$$
which extends
to an exact sequence
$$0\rightarrow\PGL_2\times\PGL_2\rightarrow\mathrm{PO}_4\rightarrow\ZZ/2\rightarrow
0.$$
In nonabelian cohomology, the map $\H^1(\Spec k,\mathrm{PO}_4)\rightarrow\H^1(\Spec
k,\ZZ/2)$ classifies the quadratic extension $\ell$ and if $\ell$ is trivial,
then the fibers give the pairs $\Ascr_1$ and $\Ascr_2$ as in
Example~\ref{ex:pairs}.

The Picard scheme $\Pic_{\PP^1\times\PP^1/S}$ is discrete and isomorphic to
the constant sheaf $\ZZ^2$. We let
$\Oscr(i,j)=p_1^*\Oscr(i)\otimes p_2^*\Oscr(j)$,
where $p_1$ is projection onto the first factor and $p_2$ projection onto the second
factor. These give all isomorphism classes of line bundles on
$\PP^1\times\PP^1$ if $S$ is the spectrum of a field.
By the theorem of Bondal and Orlov~\cite[Theorem~3.1]{bondal-orlov}, which applies since the canonical class of
$\PP^1\times\PP^1$ is antiample, it follows that there is an exact sequence
$$0\rightarrow\ZZ\times\ZZ^{\oplus
2}\rightarrow\pi_0\ShAut_{\Perfscr(\PP^1\times\PP^1)}\rightarrow\mathrm{PO}_4\rightarrow
0.$$
The rank three kernel corresponds to tensoring with line bundles and with
translation in the derived category.

Now, consider the following two filtrations on $\Perfscr(\PP^1\times\PP^1)$.
First is the $[1]\times[1]$-shaped filtration $\F_\star\Perfscr(\PP^1\times\PP^1)$
corresponding to the admissible subcategories
$$\xymatrix{
&\langle\Oscr(0,0),\Oscr(0,1)\rangle\ar@{^{(}->}[rd]&\\
\langle\Oscr(0,0)\rangle\ar@{^{(}->}[rd]\ar@{^{(}->}[ru]&&\langle(\Oscr(0,0),\Oscr(0,1),\Oscr(1,0),\Oscr(1,1)\rangle\\
&\langle\Oscr(0,0),\Oscr(1,0)\rangle\ar@{^{(}->}[ru]&
}$$
The second is the $[2]$-shaped filtration
$\G_\star\Perfscr(\PP^1\times\PP^1)$ corresponding to the semiorthogonal
decomposition
$$\langle\Oscr(0,0),\Oscr(0,1)\oplus\Oscr(1,0),\Oscr(1,1)\rangle.$$
There is a natural map
$\F_\star\Perfscr(\PP^1\times\PP^1)\rightarrow\G_\star\Perfscr(\PP^1\times\PP^1)$
over the collapse map  $[1]\times[1]\rightarrow[2]$ which sends
the vertices $(0,1)$ and $(1,0)$ to $1$.

Now, we consider the maps
$$\ShBr^{\F_\star}\rightarrow\ShBr^{\G_\star}\rightarrow\ShBr^{\PP^1\times\PP^1},$$
where we have made the evident abbreviations to cut down on notation.

\begin{proposition}
    \begin{enumerate}
        \item[{\rm (a)}] The sequence
            $\ShBr^{\G_\star}\rightarrow\ShBr^{\PP^1\times\PP^1}\rightarrow\ZZ^{\oplus
            2}$ is a fiber equivalence. If $X$ is any involution surface, the
            filtration $\G_\star\Perfscr(\PP^1\times\PP^1)$ descends to
            $\Perfscr(X)$.
        \item[{\rm (b)}] The map $\ShBr^{\F_\star}\rightarrow\ShBr^{\G_\star}$
            is a $\ZZ/2$-torsor. Thus, if $\G_\star\Dscr$ is a twisted form of
            $\G_\star\Perfscr(\PP^1\times\PP^1)$, there is a canonical
            obstruction class $o\in\H^1_\fppf(S,\ZZ/2)$ which vanishes if and
            only if the filtration $\G_\star\Dscr$ can be refined to a
            filtration $\F_\star\Dscr$ which is a twist of
            $\F_\star\Perfscr(\PP^1\times\PP^1)$.
    \end{enumerate}
\end{proposition}

\begin{proof}
    We leave the proof to the reader who should follow the lines of the proof
    of Theorem~\ref{thm:sb}.
\end{proof}

In general, if $X$ is an involution surface over $S$, the obstruction class
$\H^1_\fppf(S,\ZZ/2)$ associated to the problem of lifting the canonical
filtration $\G_\star\Perfscr(X)$ to $\F_\star\Perfscr(X)$ is precisely the class
of $T\rightarrow S$. See~\cite{auel-bernardara}.

\subsection{Descending exceptional blocks}

We want to explain how to prove the main descent theorem of
Ballard--Duncan--McFaddin in the language of this paper
(see~\cite[Theorem~2.15]{ballard-duncan-mcfaddin}).

Let $\Escr$ be an $\Oscr_S$-linear category with a $P$-shaped full exceptional collection
$\{e_p\}_{p\in P}$. Suppose that a group $G$ acts on $\Escr$. We say that the
full exceptional collection is $G$-stable if for each $g\in G$ and $e_p$ the
object $g\cdot e_p$ is in $\{e_p\}_{p\in P}$.

\begin{theorem}[\cite{ballard-duncan-mcfaddin}]\label{thm:bdm}
    Let $X$ be a smooth proper $S$-scheme, $T\rightarrow S$ a $G$-Galois fppf
    cover, $\Escr_T\subseteq\Perfscr(X_T)$ an admissible $\Oscr_T$-linear
    subcategory with a $G$-stable $P$-shaped full exceptional collection
    $\{e_p\}_{p\in P}$. Then, $\Escr_\ell$ descends to
    $\Escr\subseteq\Perfscr(X)$ and $\Escr$ admits a full twisted exceptional
    collection.
\end{theorem}

\begin{proof}
    By hypothesis, $G$ preserves $\Escr_T$ so it descends to
    $\Escr\subseteq\Perfscr(X)$ by Theorem~\ref{thm:main-disc}.
    Arguing as in~\cite[Lemma~2.12]{ballard-duncan-mcfaddin}, the $G$-orbits of
    the objects of $\{e_p\}_{p\in P}$ are in fact orthogonal exceptional
    objects. We can assume that $\{e_p\}_{p\in P}$ is in fact a single
    $G$-orbit, which is orthogonal. But, then, $\Escr\we\Perfscr(T)^n$.
    The descended version is then a twisted form of $\Perfscr(S)^n$. Any such
    admits a full twisted exceptional collection for example
    by~\cite[Theorem~2.16]{auel-bernardara} or by using the twisted Brauer
    space for $\Perfscr(S)^n$.
\end{proof}

With more work one can descend individual vector bundles by using markings. We
leave this to the reader.

\small
\bibliographystyle{amsplain}
\bibliography{sod}

\end{document}